\documentclass[a4paper,12pt]{article}

\usepackage{pdfpages}
\usepackage{graphicx}
\usepackage{mathtools}
\usepackage{verbatim}
\usepackage{amssymb}
\usepackage{bm}
\usepackage{enumitem}
\usepackage{cases}
\usepackage{latexsym}
\usepackage{amsthm}
\usepackage{amsfonts}
\usepackage{bbm}
\usepackage{etoolbox}
\usepackage[top=25mm, bottom=25mm, left=28mm, right=28mm]{geometry}
\usepackage{cite}
\usepackage{changepage}
\usepackage{hyperref}
\usepackage{wasysym}

\newtheorem{theorem}{Theorem}
\AfterEndEnvironment{theorem}{\noindent\ignorespaces}
\newtheorem{corollary}[theorem]{Corollary}
\AfterEndEnvironment{corollary}{\noindent\ignorespaces}
\newtheorem{lemma}[theorem]{Lemma}
\AfterEndEnvironment{lemma}{\noindent\ignorespaces}

\AfterEndEnvironment{proposition}{\noindent\ignorespaces}
\newtheorem*{question}{Question}
\AfterEndEnvironment{question}{\noindent\ignorespaces}

\theoremstyle{definition}

\AfterEndEnvironment{definition}{\noindent\ignorespaces}

\theoremstyle{remark}
\newtheorem*{example}{Example}
\AfterEndEnvironment{example}{\noindent\ignorespaces}
\newtheorem*{remark}{Remark}
\AfterEndEnvironment{remark}{\noindent\ignorespaces}

\AfterEndEnvironment{notation}{\noindent\ignorespaces}

\usepackage{chngcntr}
\usepackage{apptools}
\AtAppendix{\counterwithin{theorem}{section}}

\def\e{\epsilon}

\def\Z{\mathbb{Z}}

\def\F{\mathbb{F}}

\def\a{\alpha}
\def\be{\beta}

\def\g{\gamma}
\def\d{\delta}

\def\codim{\mathop{\mathrm{codim}}}

\title{The slice rank of a direct sum}

\date{}

\begin{document}

\maketitle

\abstract{We show that the slice rank of the direct sum of two tensors is equal to the sum of their slice ranks. This result generalizes the fact, shown by Tao, that the slice rank of a diagonal tensor is equal to the number of non-zero entries of that tensor. The proof uses the duality method of Sawin and Tao in a straightforward way.}

\section{Introduction}

By a $d$-\emph{tensor} over a field $\F$, we shall mean a function of the form $T:X_1\times\dots\times X_d\to\mathbb F$, where $X_1,\dots,X_d$ are finite sets. When $d=2$, we can think of $T$ as an $|X_1|\times|X_2|$ matrix, and an important invariant associated with it is its rank. It is natural to try to generalize the notion of rank to higher-order tensors, but it turns out that there are several competing generalizations, each with different advantages and disadvantages, more than one of which is genuinely useful. 

If $u_i:X_i\to\F$ for $i=1,\dots,d$, write $u_1\otimes\dots\otimes u_d$ for the tensor $T$ given by 
\[T(x_1,\dots,x_d)=u_1(x_1)u_2(x_2)\dots u_d(x_d).\]
Tensors of this form are said to have tensor rank equal to 1. Then the \emph{tensor rank} of $T$ is the smallest $r$ such $T$ is a sum of $r$ tensors of tensor rank 1. Note that when $d=2$ this definition is one of the standard ways of defining the rank of a matrix. 

A second definition of rank can be obtained by changing what we count as a rank-1 tensor. Let us say that a tensor has partition rank 1 if there is a partition of $\{1,\dots,d\}$ into non-empty sets $S_1$ and $S_2$ and $T$ splits up as a product $T=T_1T_2$, where each $T_i$ depends only on the variables $x_j$ such that $j\in S_i$. Note that for $d\geq 2$ a tensor of tensor rank 1 has partition rank 1 and that any partition of $\{1,\dots,d\}$ into two disjoint sets can be used. In general, the \emph{partition rank} of a tensor $T$ is the smallest $r$ such that $T$ is a sum of $r$ tensors of partition rank 1.

An intermediate definition is that of slice rank. Here, the tensors of rank 1 are defined as for partition rank except that we insist that $S_1$ is a singleton. So for instance if $d=4$, then a tensor of the form $u(x_1,x_2)v(x_3,x_4)$ has partition rank 1 but does not necessarily have slice rank 1, whereas a tensor of the form $u(x_3)v(x_1,x_2,x_4)$ has slice rank 1 and partition rank 1. As one would expect, the \emph{slice rank} of a tensor $T$ is the smallest $r$ such that $T$ is a sum of $r$ tensors of slice rank 1. 

Since a tensor of tensor rank 1 has slice rank 1 and a tensor of slice rank 1 has partition rank 1, we find that the tensor rank is at least as big as the slice rank, which is at least as big as the partition rank.

In a remarkable and very quick sequence of developments in 2016, Croot, Lev and Pach proved that subsets of $\Z_4^n$ that do not contain an arithmetic progression of length 3 have exponentially small density \cite{crootlevpach}, and then Ellenberg and Gijswijt proved the same for subsets of $\F_3^n$, thereby solving the famous cap-set problem in additive combinatorics \cite{ellenberggijswijt}. Soon after that, Tao gave a more conceptual reformulation of the argument \cite{tao}, in which the following lemma (in the case $d=3$) played a crucial role.

\begin{lemma}[Tao]\label{diagonal} Let $T:X^d\to\F$ be a $d$-tensor and suppose that $T(x_1,\dots,x_d)=0$ except if $x_1=x_2=\dots=x_d$. Then the slice rank of $T$ is equal to the number of non-zero entries of $T$. 
\end{lemma}

We briefly sketch his proof in the case $d=3$. Suppose that one has a decomposition
\[T(x,y,z)=\sum_{i=1}^ra_i(x)b_i(y,z)+\sum_{j=1}^sc_j(y)d_j(x,z)+\sum_{k=1}^te_k(z)f_k(x,y).\]
Then a simple linear algebra argument shows that there is a function $h:X\to\F$ such that $\sum_xh(x)a_i(x)=0$ for $i=1,\dots,r$ and such that $h(x)=0$ for at most $r$ values of $x$. Take such an $h$ and consider the matrix $M(y,z)=\sum_xh(x)T(x,y,z)$. Then $M$ is diagonal, and $M(y,y)=h(y)T(y,y,y)$. If the number of non-zero entries of $T$ is $m$, then the number of non-zero entries of $M$ is at least $m-r$, so $M$ has rank at least $m-r$.

On the other hand, $M$ has a decomposition
\[M(y,z)=\sum_{j=1}^sc_j(y)u_j(z)+\sum_{k=1}^tv_k(y)e_k(z),\]
where $u_j(z)=\sum_xh(x)d_j(x,z)$ and $v_k(y)=\sum_xh(x)f_k(x,y)$ for each $j,k$. It follows that $M$ has rank at most $s+t$. 

Putting these two estimates together, we deduce that $m-r\leq s+t$. Since the initial decomposition of $T$ was arbitrary, this proves that the slice rank of $T$ is at least $m$, as we wanted. 
\smallskip

In this paper, we shall prove the following result. Suppose we have finite sets $X_1,\dots,X_d$ and for each $i$ let $X_i=X_i^1\cup X_i^2$, where this is a disjoint union. 
Given two tensors $T_i:X_1^i\times\dots\times X_d^i\to\F$, $i=1,2$, their \emph{direct sum} $T_1\oplus T_2$ is the tensor that takes the value $T_1(x_1,\dots,x_d)$ if $x_i\in X_i^1$ for each $i$, $T_2(x_1,\dots,x_d)$ if $x_i\in X_i^2$ for each $i$, and 0 otherwise. 

Let us write $\sigma(T)$ for the slice rank of $T$. 

\begin{theorem}\label{main} For any two tensors, we have $\sigma(T_1\oplus T_2)=\sigma(T_1)+\sigma(T_2)$.
\end{theorem}

\noindent Note that this immediately implies that $\sigma(T_1\oplus\dots\oplus T_m)=\sigma(T_1)+\dots+\sigma(T_m)$ (where the definition of $T_1\oplus\dots\oplus T_m$ is obvious), and hence Tao's lemma, which is the special case where each $T_i$ is a $1\times\dots\times 1$ tensor. 

To prove the theorem, it is tempting to try to modify Tao's argument, but the following example, with $d=3$, seems to indicate that that cannot be done straightforwardly.

\begin{example}
Let $\epsilon$ be the $3\times 3\times 3$ Levi-Civita symbol. That is, it is defined on $\{1,2,3\}^3$, and $\epsilon(x,y,z)=0$ if any two of $x,y,z$ are equal, and otherwise $\epsilon(x,y,z)=1$ if $(x,y,z)$ is an even permutation of $(1,2,3)$ and $-1$ if it is an odd permutation. (It would more normally be written $\epsilon_{ijk}$, but we write it $\e(x,y,z)$ for consistency with our earlier notation.) This tensor is supported on an antichain, meaning that if $x\leq x', y\leq y', z\leq z'$, and both $(x,y,z)$ and $(x',y',z')$ belong to the support, then $(x,y,z)=(x',y',z')$. If we define a \emph{slice} to be a subset of $\{1,2,3\}^3$ defined by holding one of the coordinates constant, then the number of slices needed to cover the support of $\e$ is 3, since each slice contains two points of the support. A result of Sawin and Tao \cite{sawintao} states that if a tensor is supported on an antichain, then its slice rank is equal to the number of slices needed to cover the support, which implies that $\e$ has slice rank~3. 

If, however, $h$ is any function from $\{1,2,3\}$ to $\F$, then the $3\times 3$ matrix
\[M(y,z)=\sum_xh(x)\e(x,y,z)\]
is antisymmetric, and therefore has rank at most 2.

To see why this is a problem, let $T=\e\oplus\dots\oplus\e$, where we take $m$ copies, and suppose we have a decomposition
\[T(x,y,z)=\sum_{i=1}^ra_i(x)b_i(y,z)+\sum_{j=1}^sc_j(y)d_j(x,z)+\sum_{k=1}^te_k(z)f_k(x,y).\]
We can find $h$ with at most $r$ zeros such that $\sum_xh(x)a_i(x)=0$ for $i=1,2,\dots,r$, and the matrix 
\[M(y,z)=\sum_xh(x)T(x,y,z)\]
has rank at most $s+t$. 

However, in the other direction all we know is that the rank of $M$ is twice the number of copies of $\e$ that are not projected to zero -- that is, twice the number of $q$ such that at least one of $h(3q-2), h(3q-1)$ and $h(3q)$ is non-zero. The number of such $q$ is at least $m-\lfloor r/3\rfloor$, but can in principle be that low. For example, if for each $i\leq r$, $a_i$ is the $i$th standard basis vector, then $h(i)$ is forced to be zero for $i=1,\dots,r$, so for $q\leq r/3$ we have that $h(3q-2), h(3q-1)$ and $h(3q)$ are all zero. So the best lower bound we can obtain in general is that $2\lfloor r/3\rfloor+s+t\geq 2m$. By symmetry we obtain similar estimates with the role of $r$ played by $s$ and $t$. But if $r,s$ and $t$ are all equal and are multiples of 3, then we find that $8(r+s+t)/9\geq 2m$, from which we can conclude only that $r+s+t\geq 9m/4$. 

Note that the result of Sawin and Tao that shows that $\sigma(\e)=3$ also shows that $\sigma(\e\oplus\dots\oplus\e)=3m$ (where there are still $m$ copies of $\e$), but there are tensors for which their method does not give optimal estimates, so this argument will only work for special cases of the problem.
\end{example}

\begin{remark}
The example just presented relied on a ``non-trivial" space of low-rank matrices, namely the $3\times 3$ antisymmetric matrices. We regard a space $Z$ of matrices of rank at most $r$ as trivial if there are spaces $U$ and $V$ of dimensions $s$ and $t$ with $s+t\leq r$ such that $Z$ is the sum of the space of matrices with rows in $U$ and the space of matrices with columns in $V$. It is not a straightforward problem to understand spaces of low-rank matrices in general. See for example a paper of Eisenbud and Harris \cite{eisenbudharris}, which was what led us to think of the example above, and which can probably be used to construct other examples of a similar type.
\end{remark}

An earlier version of this note contained a more complicated argument. I would like to thank Thomas Karam for pointing out that certain parts of that argument were imprecise to the point of not being obviously correct. Although it turned out that the argument could be rescued in the case $d=3$ (and probably also for general $d$ but that is trickier), during subsequent conversations with Thomas Karam a simpler proof emerged, after which it became clear that the result could in fact be proved using a simple modification of the argument of Sawin and Tao just mentioned, a possibility that I had previously considered but, as a result of an incorrect heuristic argument, discounted. While this makes the result not interesting enough to publish formally, it still seems worth keeping it as an arXiv preprint, since at some point it may save somebody some time if it can be readily found online. As this document is not intended for publication, we include the modified old argument for the case $d=3$, just in case elements of the proof are of use to anyone. 

\section{Proof of Theorem \ref{main}}

For the convenience of the reader, we begin by recalling one or two facts from a blog post of Sawin and Tao \cite{sawintao}. The first is that we can think of tensors in two different ways -- either in ``matrix form" as functions $T:X_1\times\dots\times X_d\to\mathbb F$ or as elements of a tensor product $V_1\otimes\dots\otimes V_d$. Given a function $T:X_1\otimes\dots\otimes X_d\to\F$, the corresponding element of the tensor product $\F^{X_1}\otimes\dots\otimes\F^{X_d}$ is the sum
\[\sum_{x_1,\dots,x_d}T(x_1,\dots,x_d)e_{x_1}\otimes\dots\otimes e_{x_d},\]
where, given $x_i\in X_i$, the vector $e_{x_i}$ is the standard basis vector in $\F^{X_i}$ that takes the value 1 at $x_i$ and 0 everywhere else. In the other direction, given an element $\tau$ of a tensor product $V_1\otimes\dots\otimes V_d$ of finite-dimensional vector spaces, take a basis $\{e_{i1},\dots,e_{ir_i}\}$ of each $V_i$, write $\tau$ in the unique way possible as
\[\tau=\sum_{j_1,\dots,j_d}\lambda(j_1,\dots,j_d)e_{1j_1}\otimes\dots\otimes e_{dj_d},\]
let $X_i=\{1,2,\dots,r_i\}$, and set $T=\lambda$. 

In the tensor-product formulation, the slice rank of a tensor $T\in V_1\otimes\dots\otimes V_d$ is the smallest $r$ such that it is possible to write $T$ in the form 
\[\sum_{i=1}^d\sum_{j=1}^{r_i}u_{ij}\otimes v_{ij},\]
with $r_1+\dots+r_d=r$, where for each $i$, $u_{ij}\in V_i$ and $v_{ij}\in V_1\otimes\dots\otimes V_{i-1}\otimes V_{i+1}\otimes\dots\otimes V_d$. (This is a slight abuse of notation because what we are really doing is ``inserting" $u_{i,j}$ into $v_{i,j}$. More precisely, if $v$ is a pure tensor $w_1\otimes\dots\otimes w_{i-1}\otimes w_{i+1}\otimes\dots\otimes w_d$ and $u\in V_i$, then by $u\otimes v$ we mean the tensor $w_1\otimes\dots\otimes w_{i-1}\otimes u\otimes w_{i+1}\otimes\dots\otimes w_d$, and then this map can be extended linearly.) 

It is simple to check that this tensor-product definition of slice rank agrees with the definition given earlier. We shall therefore pass freely between the two, using whichever formulation is more convenient at any one moment.

\begin{lemma} Let $V_1,\dots,V_d$ be finite-dimensional vector spaces over a field $\F$ and let $T\in V_1\otimes\dots\otimes V_d$. Then $T$ has slice rank at most $r$ if and only if there exist subspaces $U_i\subset V_i^*$ with $\sum_i\codim(U_i)\leq r$ such that $\langle T,u\rangle=0$ for every $u\in U_1\otimes\dots\otimes U_d$. 
\end{lemma}

\begin{proof}
Suppose first that $T$ has slice rank at most $r$. Then we can write $T$ as a sum $\sum_{i=1}^d\sum_{j=1}^{r_i}v_{ij}\otimes w_{ij}$, where for each $i$ and $j$, $v_{ij}\in V_i$ and $w_{ij}\in V_1\otimes\dots\otimes V_{i-1}\otimes V_{i+1}\otimes\dots\otimes V_d$, and $\sum_{i=1}^dr_i\leq r$.

For each $i$ let $U_i$ be the set of all $u\in V_i^*$ such that $\langle v_{ij},u\rangle=0$ for $j=1,\dots,r_i$. Then $U_i$ is a subspace of codimension at most $r_i$. Moreover, if $u_i\in U_i$ for $i=1,\dots,d$, then $\langle\sum_{i=1}^d\sum_{j=1}^{r_i}v_{ij}\otimes w_{ij},u_1\otimes\dots\otimes u_d\rangle=0$, since for each $i,j$ we have that $\langle v_{ij},u_i\rangle=0$. Extending linearly we find that $\langle T,u\rangle=0$ for every $u\in U_1\otimes\dots\otimes U_d$, and we also have that $\sum_i\codim(U_i)\leq r$. 

In the reverse direction, suppose that such subspaces $U_i$ exist. For each $i$ choose a basis of $U_i$ and extend it to a basis of $V_i^*$. By considering the expansion of $T$ with respect to the dual bases of these bases, we see that $T$ must be contained in the subspace $\sum_{i=1}^d V_1\otimes\dots\otimes V_{i-1}\otimes U_i^\perp\otimes V_{i+1}\otimes\dots\otimes V_d$ of $V_1\otimes\dots\otimes V_d$. Since $U_i^\perp$ has a basis of size $r_i$, this yields a decomposition of $T$ of the required form. 
\end{proof}

\begin{proof}[Proof of Theorem \ref{main}]
Let $V_1,\dots,V_d$ be finite-dimensional vector spaces with $V_i=V_i^1\oplus V_i^2$, and let $T=T_1+T_2$, where $T^1\in V_1^1\otimes\dots\otimes V_d^1$ and $T^2\in V_1^2\otimes\dots\otimes V_d^2$. We would like to show that $\sigma(T_1)+\sigma(T_2)\leq\sigma(T)$, the reverse inequality being trivial. 

Let $r=\sigma(T)$ and choose subspaces $U_i\subset V_i^*$ with $\codim(U_i)=r_i$ and $\sum_ir_i=r$ such that $\langle T,u\rangle=0$ for every $u\in U_1\otimes\dots\otimes U_d$.  

For each $i$, choose a basis $v_{i1},\dots,v_{in_i}$ of $V_i$ that starts with a basis of $V_i^1$ and ends with a basis of $V_i^2$. Let $v_{i1}^*,\dots,v_{in_i}^*$ be the dual basis, and let $u_{i1},\dots,u_{im_i}$ be a basis of $U_i$, where $m_i=n_i-r_i$. Each $u_{ij}$ can be expanded in terms of the dual basis. Let us write $u_{ij}(h)$ for the $h$th coefficient of $u_{ij}$ with respect to this basis: that is,
\[u_{ij}=\sum_{h=1}^{n_i}u_{ij}(h)v_{ih}.\] 
By applying Gaussian elimination, we may assume for any given $i$ that the first $h$ for which $u_{ij}(h)$ is non-zero is a strictly increasing function of $j$. Alternatively, we may assume for any given $i$ that the last $h$ for which $u_{ij}(h)$ is non-zero is a strictly decreasing function of $i$. (That is, for each $i$ we may assume one or the other of these two statements: we do not claim that both can be assumed at once.)

Suppose that for a particular $i$ we have chosen the first option: that is, the first $h$ for which $u_{ij}(h)\ne 0$ is strictly increasing with $j$. If $\dim(V_i^1)=s_i$, then for every $j$ such that the first such $h$ is greater than $s_i$, we have that $u_{ij}$ vanishes on $V_i^1$. We now define a sequence $w_{i1},\dots,w_{im_i}$ as follows. For every $j$ such that the first $h$ is less than $s_i$, we let $w_{ij}$ be the projection of $u_{ij}$ on to the first $s_i$ coordinates, and note that $w_{ij}$ and $u_{ij}$ agree on $V_i^1$. Let the number of such $j$ be $k_i$. For every $j>k_i$, we let $w_{ij}=u_{ij}$, and as just mentioned we have that $w_{ij}$ vanishes on $V_i^1$.

Similarly, if we have chosen the second option, then we can define a sequence $w_{i1},\dots,w_{im_i}$ and $k_i$ such that for $j\leq k_i$ we have that $w_{ij}$ vanishes on $V_i^2$ and for $j>k_i$ we have that $w_{ij}$ agrees with $u_{ij}$ on $V_i^2$. 

In both cases we start with the vectors $u_{i1},\dots,u_{im_i}$ and obtain a sequence $w_{i1},\dots,w_{im_i}$ and some $k_i$ such that $w_{i1},\dots,w_{ik_i}\in (V_i^1)^*$ and $w_{i,k_i+1},\dots,w_{im_i}\in (V_i^2)^*$. For each $i$ let $U_i^1$ be the span of $w_{i1},\dots,w_{ik_i}$ and let $U_i^2$ be the span of $w_{i,k_i+1},\dots,w_{im_i}$. Since $\dim(U_i^1)+\dim(U_i^2)=m_i$, we have that $\codim(U_i^1)+\codim(U_i^2)=n_i-m_i=r_i$. (Here by the codimension of $U_i^1$ we mean its codimension as a subspace of $(V_i^1)^*$, and similarly for $U_i^2$.) 

Assume now that there exists $i_0$ such that the second option is chosen. We claim that if $u\in U_1^1\otimes\dots\otimes U_d^1$, then $\langle T^1,u\rangle=0$. It is enough to prove this when $u=u_1\otimes\dots\otimes u_d$ with $u_i\in U_i^1$. Furthermore, it is enough to prove it when each $u_i$ is equal to $w_{ij}$ for some $j\leq k_i$. 

If $i$ is such that the first option is chosen, and $j\leq k_i$, then $w_{ij}$ agrees with $u_{ij}$ on $V_i^1$. If $i$ is such that the second option is chosen, and $j\leq k_i$, then $w_{ij}=u_{ij}$ and therefore also agrees with $u_{ij}$ on $V_i^1$. Since $T^1\in V_1^1\otimes\dots\otimes V_d^1$, it follows that $\langle T^1,u\rangle$ does not change if we replace each $w_{ij}$ by $u_{ij}$. But if each $u_i$ is one of the vectors $u_{ij}$, then $\langle T,u\rangle=0$, by hypothesis. Also, since the second option is chosen for $i_0$, $u_{i_0}$ vanishes on $V_{i_0}^2$. It follows that $\langle T^2,u\rangle=0$, and therefore that $\langle T^1,u\rangle=0$. 

Similarly, if $u\in U_1^2\otimes\dots\otimes U_d^2$ and we choose the first option for at least one $i$, then $\langle T^2,u\rangle=0$.

Since $d\geq 2$, we can choose the first option for at least one $i$ and the second option for at least one $i$, so the result is proved.
\end{proof}

An examination of the above argument shows that it can be used to prove stronger statements as well. Suppose, for instance, that $T$ is of the form $T^1+T^2$ where $T^1$, as before, belongs to $V_1^1\otimes\dots\otimes V_d^1$, but all we assume about $T^2$ is that it belongs to $V_1\otimes\dots\otimes V_{d-1}\otimes V_d^2$. We now run the proof, choosing the first option for $i=1,2,\dots,d-1$ and the second option for $i=d$. 

Suppose that $w=w_1\otimes\dots\otimes w_d$, where each $w_i$ is equal to $w_{ij}$ for some $j\leq k_i$. For $i=1,2,3,\dots,d-1$ let us replace $w_i$ by some $u_i\in U_i^1$ that agrees with $w_i$ on $V_i^1$. Letting $u=u_1\otimes\dots\otimes u_{d-1}\otimes w_d$, we then have that $\langle T^1,w\rangle=\langle T^1,u\rangle$. Because we chose the second option for $i=d$, $w_d$ vanishes on $V_d^2$, and therefore $\langle T^2,u\rangle=0$. It follows that $\langle T^1,u\rangle=\langle T,u\rangle=0$, where the last equality holds by hypothesis.

Now suppose that $w=w_1\otimes\dots\otimes w_d$, where this time each $w_i$ is equal to $w_{ij}$ for some $j>k_i$. Then there exists $u_d\in U_d^2$ that agrees with $w_d$ on $V_d^2$, while for $i=1,2,\dots,d-1$ we have that $w_i$ vanishes on $V_i^1$. Let $u=w_1\otimes\dots\otimes w_{d-1}\otimes u_d\in U_1^2\otimes\dots\otimes U_d^2$ and note that $\langle T^1,u\rangle=0$. 

Given $\alpha\in\{1,2\}^d$, let $T^\a$ stand for the projection of $T$ to $V_1^{\a_1}\otimes\dots\otimes V_d^{\a_d}$. (To be more explicit, given $v_i\in V_i$ we can write it uniquely as $v_i^1+v_i^2$ with $v_i^1\in V_i^1$ and $v_i^2\in V_i^2$. This allows us to decompose $v_1\otimes\dots\otimes v_d\in V_1\otimes\dots\otimes V_d$ into $2^d$ parts $v_1^{\a_1}\otimes\dots\otimes v_d^{\a_d}$, one for each $\a$.) Then $T=\sum_\a T^\a$. 

If any of $\a_1,\dots,\a_{d-1}$ is equal to $1$, then because $w_i$ vanishes on $V_i^1$ for $i\leq d-1$, we have that $\langle T^{\a},w\rangle=0$. Also, when $\a=(2,2,\dots,2,1)$, we have that $T^\a=0$. It follows that 
\[\langle T^{22\dots 2},w\rangle=\langle T^2,w\rangle=\langle T^2,u\rangle=\langle T,u\rangle=0,\]
where again the last equality holds by hypothesis. 

This proves the following statement.

\begin{theorem}\label{general}
Let $V_1,\dots,V_d$ be finite-dimensional vector spaces with $V_i=V_i^1\oplus V_i^2$ for each $i$. Let $T\in V_1\otimes\dots\otimes V_d$ and suppose that the component $T^\a$ (see just above for the definition) is zero unless either $\a_d=2$ or $\a_1=\dots=\a_d=1$. Then 
\[\sigma(T)\geq\sigma(T^{11\dots 1})+\sigma(T^{22\dots 2}).\]
\end{theorem}

Note that the conditions of this theorem are satisfied in particular if $T^\a$ is non-zero only for increasing sequences $\a$. This gives us a simple corollary about ``block upper triangular" tensors. Here we let $T\in V_1\otimes\dots\otimes V_d$ as before, but this time $V_i=V_i^1\oplus\dots\oplus V_i^k$ for some $k$. We call a tensor block upper triangular (with respect to the given decompositions) if the component $T^\a$ (defined in the obvious way for each $\a\in[k]^d$) is non-zero only for increasing sequences $\a$. 

\begin{corollary}\label{triangular}
Let $V_1,\dots,V_d$ be as above and let $T\in V_1\otimes\dots\otimes V_d$ be upper triangular. Then 
\[\sigma(T)\geq\sigma(T^{11\dots 1})+\dots+\sigma(T^{kk\dots k}).\]
\end{corollary}

\begin{proof}
For each $i$ let $W_i^1=V_i^1\oplus\dots\oplus V_i^{k-1}$ and let $W_i^2=V_i^k$. Then $V_i=W_i^1\oplus W_i^2$. For $\a\in\{1,2\}^d$ let $S^\a$ be the component of $T$ in $W_1^{\a_1}\otimes\dots\otimes W_d^{\a_d}$. Then $T$ is block upper triangular with respect to the decompositions $V_i=W_i^1\oplus W_i^2$, from which it follows, using the theorem just proved, that $\sigma(T)\geq\sigma(S^{11\dots 1})+\sigma(S^{22\dots 2})$. 

But $S^{22\dots 2}=T^{22\dots 2}$, and $S^{11\dots 1}\in W_1^1\otimes\dots\otimes W_d^1$ is block upper triangular with respect to the decompositions $W_i^1=V_i^1\oplus\dots\oplus V_i^{k-1}$. By induction on $k$ we have that 
\[\sigma(S^{11\dots 1})\geq\sigma(T^{11\dots 1})+\dots+\sigma(T^{k-1,k-1,\dots,k-1}),\]
and the proof is complete.
\end{proof}

\iftrue
\else
\section{An alternative proof of Theorem \ref{main} when $d=3$}

There seems no harm in including the argument mentioned earlier that works when $d=3$, even though it is a little more complicated, as the lemmas along the way may be of some interest.

We begin with a lemma that will allow us to assume that various functions are linearly independent.

\begin{lemma}\label{rewriteterm} Let $A$ and $B$ be vector spaces, let $a_1,\dots,a_r$ and $a_1',\dots,a_s'$ be vectors that generate the same subspace of $A$, and let $b_1,\dots,b_r$ be vectors in $B$. Then there exist vectors $b_1',\dots,b_s'$ such that $\sum_ia_i\otimes b_i=\sum_ja_j'\otimes b_j'$.
\end{lemma}

\begin{proof}
Let $a_i=\sum_{j=1}^r\theta_{ij}a_j'$ for each $i$, which we can do because the $a_j'$ contain the $a_i$ in their linear span. Then
\[\sum_ia_i\otimes b_i=\sum_{i,j}\theta_{ij}a_j'\otimes b_i=\sum_j a_j'\otimes(\sum_i\theta_{ij}b_i),\]
so we can take $b_j'=\sum_i\theta_{ij}b_i$ for each $j$. 
\end{proof}

It follows immediately that if $T$ is a $d$-tensor given by a formula of the kind
\[T(x_1,\dots,x_d)=\sum_ia_i(x_1)b_i(x_2,\dots,x_d),\]
then we can replace the expression on the right by one in which the $a_i$ are linearly independent. 

Suppose now that we have a $d$ tensor $T$ of slice rank $r$ and that $T$ is given by the expression
\[T(x_1,\dots,x_d)=\sum_{i=1}^d\sum_{j_i=1}^{r_i}a_{ij_i}(x_i)b_{ij_i}(\overline{x_i}),\]
where $r_1+\dots+r_d=r$ and we write $\overline{x_i}$ as shorthand for $(x_1,\dots,x_{i-1},x_{i+1},\dots,x_d)$. By the remark just made, for each $i$ the functions $a_{i1},\dots,a_{ir_i}$ must be linearly independent. 

%The next lemma is equivalent to the statement that matrix rank is additive. It is thus the case $d=2$ of the main result.

%\begin{lemma} \label{rankadditivity}
%Let $U$ and $V$ be vector spaces with $U=U_1\oplus U_2$ and $V=V_1\oplus V_2$. Let $\sum_{i=1}^ru_i\otimes v_i$ be an element of $U\otimes V$ that belongs to $U_1\otimes V_1+U_2\otimes V_2$. Then there exist vectors $w_1,\dots,w_r\in U$ and $z_1,\dots,z_r\in V$ such that $\sum_{i=1}^rw_i\otimes z_i=\sum_{i=1}^ru_i\otimes v_i$ and such that for each $i$, $w_i\otimes z_i$ belongs either to $U_1\otimes V_1$ or to $U_2\otimes V_2$.
%\end{lemma}

If $R$ and $S$ are two tensors, we shall use the notation $R.S$ for the tensor multiplication where we sum over all variables that $R$ and $S$ have in common and not over the others. For example, if $R$ is a function of $x_1$ and $x_2$ and $S$ is a function of $x_2$ and $x_3$, then
\[R.S(x_1,x_3)=\sum_{x_2}R(x_1,x_2)S(x_2,x_3),\]
so in this case we have matrix multiplication. We shall make particular use of the case where $R$ is a one-variable function: if $R$ is a function of $x_i$ and $S$ is a function of $x_1,\dots,x_d$, then 
\[R.S(\overline{x_i})=\sum_{x_i}R(x_i)S(x_1,\dots,x_d).\]

\begin{lemma}\label{triangulardecomposition}
Let $T$ be a $d$-tensor given by the formula
\[T(x_1,\dots,x_d)=\sum_{i=1}^d\sum_{j_i=1}^{r_i}a_{ij_i}(x_i)b_{ij_i}(\overline{x_i}),\]
and suppose that for each $i$ the functions $a_{ij_i}$ are linearly independent. For each $i$ choose dual functions $a_{ij_i'}^*$ such that $a_{ij_i'}^*.a_{ij_i}=\d_{j_ij_i'}$ for every $j_i,j_i'\in[r_i]$. Then there are functions $b_{ij_i'}$ such that
\[T(x_1,\dots,x_d)=\sum_{i=1}^d\sum_{j_i=1}^{r_i}a_{ij_i}(x_i)b_{ij_i}'(\overline{x_i})\]
and such that $a_{sj_s}^*.b_{tj_t}'=0$ whenever $1\leq s<t\leq d$. 
\end{lemma}

\begin{proof}
For each $i$ let $P_i$ be the projection $T\mapsto\sum_{j_i}a_{ij_i}\otimes(a_{ij_i}^*.T)$. That is, 
\[P_iT(x_1,\dots,x_d)=\sum_{j_i=1}^{r_i}a_{ij_i}(x_i)\sum_{x_i'}a_{ij_i}(x_i')T(x_1,\dots,x_{i-1},x_i',x_{i+1},\dots,x_d).\]
It is easy to check that this is indeed a projection -- that is, that $P_i^2=P_i$. Write $Q_i$ for $I-P_i$. It is also straightforward to see that the $P_i$ commute. We also have that for any tensor $S$, $P_iS$ is a tensor given by a formula of the form $\sum_{j_i}a_{ij_i}(x_i)b_i'(\overline{x_i})$. 

By a simple induction.
\[T=P_1T+P_2Q_1T+\dots+P_dQ_{d-1}\dots Q_1T+Q_dQ_{d-1}\dots Q_1T\]
for any tensor $T$, so by the above remark we will be done provided we can prove that $Q_dQ_{d-1}\dots Q_1T=0$ and that the last condition of the statement of the theorem is satisfied.

Since the $P_i$ commute, so do the $Q_i$, and any function given by a formula of the form $\sum_{j_i=1}^{r_i}a_{ij_i}(x_i)b_{ij_i}(\overline{x_i})$ belongs to the kernel of $Q_i$. It follows that $Q_dQ_{d-1}\dots Q_1T=0$ as required.

Finally, if $s<t$ we have that $P_s(P_tQ_{t-1}\dots Q_1T)=0$, since the projections commute and $P_sQ_s=0$. From this it follows that $a_{sj_s}.b_{tj_t}'=0$ for each $j_s$ and $j_t$.
\end{proof}

It is not hard to prove that the decomposition above is unique, but we shall not need this. 

\begin{proof}[Proof of Theorem \ref{main}]
For $i=1,2$ let $T^i:X_1^i\times\dots\times X_d^i\to\F$ be a $d$-tensor and let $T=T_1\oplus T_2$. Suppose that $T$ has slice rank $r=r_1+\dots+r_d$ and that $T$ is given by the formula
\[T(x_1,\dots,x_d)=\sum_{i=1}^d\sum_{j_i=1}^{r_i}a_{ij_i}(x_i)b_{ij_i}(\overline{x_i}),\]
By Lemma \ref{triangulardecomposition} we may assume that $a_{rj_r}^*.b_{sj_s}=0$ whenever $r<s$.

From this it follows that for every $j_1\in[r_1]$ we have
\[a_{1j_1}^*.T(\overline{x_1})=b_{1j_1}(\overline{x_1}).\]
But for any function $a^*:X_1\to\F$, we know that 
\[a^*.T=a^*.(T^1\oplus T^2)=a^*.T^1\oplus a^*.T^2,\]
where the last equality holds because if $(x_2,\dots,x_d)\in X_2^{\a_2}\times\dots\times X_d^{\a_d}$ and the $\a_i$ are not all equal, then $T(x_1,x_2,\dots,x_d)=0$ for every $x_1$. 

This tells us that $b_{1j_1}$ can be written as a sum $b_{1j_1}^1+b_{1j_1}^2$, where $b_{1j_1}^1$ is supported in $X_2^1\times\dots\times X_d^1$ and $b_{2j_2}$ is supported in $X_2^2\times\dots\times X_d^2$. Writing $X^\a$ for the set $X_1^{\a_1}\times\dots\times X_d^{\a_d}$, it follows that $P_1T$ is supported in the set
\[X^{111\dots 1}\cup X^{211\dots 1}\cup X^{122\dots 2}\cup X^{222\dots 2}.\]
Essentially the same argument for each $s$, $P_sQ_{s-1}\dots Q_1T$ is supported in the union of the four sets $X^{\a}$ for which $\a_1=\dots=\a_{s-1}=\a_{s+1}=\dots=\a_d$, since 
\[Q_{s-1}\dots Q_1T(x_1,\dots,x_d)=\sum_{i=s}^d\sum_{j_i=1}^{r_i}a_{ij_i}(x_i)b_{ij_i}(\overline{x_i}).\]
But if $s\ne t$ and $\a$ and $\be$ are two non-constant sequences in $\{1,2\}^d$ such that $\a_1=\dots=\a_{s-1}=\a_{s+1}=\dots=\a_d$ and $\be_1=\dots=\be_{t-1}=\be_{t+1}=\dots=\be_d$, then $\a\ne\be$, so $X^\a$ and $X^\be$ are disjoint. Since $T$ is supported in $X^{11\dots 1}\cup X^{22\dots 2}$, it follows that all the tensors $P_sQ_{s-1}\dots Q_1T$ are as well.

We may therefore regard $P_iT$ as a matrix of rank $r_i$ indexed by $X_i^1\cup X_i^2$ and 
\[X_1^1\times\dots\times X_{i-1}^1\times X_{i+1}^1\times\dots\times X_d^1\cup X_1^2\times\dots\times X_{i-1}^2\times X_{i+1}^2\times\dots\times X_d^2,\]
which splits up into four blocks and is supported in the two diagonal blocks. Since matrix rank is additive, it follows that we can find functions $c_{ij_i}$ and $d_{ij_i}$ such that
\[\sum_{j_i}a_{ij_i}(x_i)b_{ij_i}(\overline{x_i})=\sum_{j_i}c_{ij_i}(x_i)d_{ij_i}(\overline{x_i}),\]
where for each $j_i$ there exists $\a\in\{1,2\}$ such that $c_{ij_i}$ is supported in $X_i^\a$ and $d_{ij_i}$ is supported in $X_1^\a\times\dots\times X_{i-1}^\a\times X_{i+1}^\a\times\dots\times X_d^\a$.

This gives us that
\[T(x_1,\dots,x_d)=\sum_i\sum_{j_i=1}^{r_i}c_{ij_i}(x_i)d_{ij_i}(\overline{x_i}),\]
and that each term in the decomposition is supported either in $X^{11\dots 1}$ or in $X^{22\dots 2}$. It follows that $\sigma(T_1)+\sigma(T_2)\leq\sigma(T)$, as required.
\end{proof}

\end{document}
\fi

\section{An alternative proof of Theorem \ref{main} for 3-tensors}

There seems no harm in including the argument mentioned earlier that works when $d=3$, even though it is a little more complicated, as the lemmas along the way may be of some interest. However, the reader just interested in obtaining some proof of Theorem \ref{main} can safely skip this section.

For this proof we shall use the more ``matrix-like" conception of tensors.

\begin{lemma}\label{independence} Let $V$ and $W$ be two vector spaces with $V\cap W=\{0\}$, let $v_1,\dots,v_n\in V$ and $w_1,\dots,w_n\in W$ be two sequences of vectors, and let $U\subset V+W$ be the subspace generated by the vectors $v_i+w_i$. Then there exists a sequence $v_1''+w_1'',\dots,v_n''+w_n''$ that generates $U$ with each $v_i''$ in $V$ and each $w_i''$ in $W$, such that the non-zero $v_i''$ are linearly independent and the non-zero $w_i''$ are linearly independent.  
\end{lemma}

\begin{proof}
Without loss of generality $v_1,\dots,v_m$ is a maximal linearly independent subset of $v_1,\dots,v_n$. Then for each $j>m$ we can write
\[v_j=\sum_{i=1}^m\lambda_{ji}v_i\]
For $j>m$ let $w_j'=w_j-\sum_{i=1}^r\lambda_{ji}w_i$ and let $v_j'=0$, and observe that the $v_i'+w_i'$ generate the same subspace as the $v_i+w_i$. (We let $v_i'=v_i$ and $w_i'=w_i$ when $i\leq m$.) We also have that the non-zero $v_i$ are linearly independent. 

Now let us choose $v_1'',\dots,v_n''$ and $w_1'',w_2'',\dots,w_n''$ as follows, with the aim of ensuring that for every $s$ we have that 
\[\langle w_s'',w_{s+1}'',\dots,w_n''\rangle=\langle w_s',w_{s+1}',\dots,w_n'\rangle.\]

We start by setting $w_n''=w_n'$. Once we have chosen $w_{s+1}'',\dots,w_n''$ with the desired property, if
\[w_s'=\sum_{s+1}^n\mu_{si}w_i''\]
then we set $w_s''=0$ and $v_s''=v_s'-\sum_{s+1}^n\mu_{si}v_i''$. Otherwise -- that is, if $w_s'$ is not a linear combination of $w_{s+1}'',\dots,w_n''$ -- we set $w_s''=w_s'$ and $v_s''=v_s'$.

Since $v_{m+1}'=\dots=v_n'=0$, we find that $v_{m+1}''=\dots=v_n''=0$ as well. Also, the non-zero $w_i''$ are linearly independent, as are the vectors $v_1'',\dots,v_m''$, and the vectors $v_i''+w_i''$ generate the same subspace as the vectors $v_i+w_i$. 
\end{proof}

In the next lemma, we write $a\otimes b$ for the function that takes the value $a(x)b(y,z)$ at $(x,y,z)$. Note that the lemma is really about matrices -- the fact that the $b_i$ are functions of two variables is irrelevant, but it is the case we shall use when we apply the lemma.

\begin{lemma}\label{rewriteterm} If $a_1,\dots,a_r$ and $a_1',\dots,a_r'$ generate the same subspace, then any tensor $\sum_ia_i(x)b_i(y,z)$ is equal to some tensor $\sum_ja_j'(x)b_j'(y,z)$.
\end{lemma}

\begin{proof}
Let $a_i=\sum_{j=1}^r\theta_{ij}a_j'$ for each $i$, which we can do because the $a_j'$ contain the $a_i$ in their linear span. Then
\[\sum_ia_i\otimes b_i=\sum_{i,j}\theta_{ij}a_j'\otimes b_i=\sum_j a_j'\otimes(\sum_i\theta_{ij}b_i),\]
so we can take $b_j'=\sum_i\theta_{ij}b_i$ for each $j$. 
\end{proof}

\begin{remark}
The lemma just proved highlights the main difference, for this question, between slice rank and tensor rank, and indeed various other kinds of rank. Each $b_j'$ is a linear combination of the $b_i$, and is therefore a function of the same type. But if we were considering tensor rank, then each $b_i$ would be a rank-1 matrix, and we would not be able to conclude that each $b_j'$ was a rank-1 matrix. Thus, there is a flexibility associated with slice-rank decompositions that we do not have with tensor-rank decompositions.
\end{remark}
\smallskip

We now take three finite sets $X$, $Y$, and $Z$, each partitioned into two subsets, so $X=X^1\cup X^2$, $Y=Y^1\cup Y^2$ and $Z=Z^1\cup Z^2$. (We shall use superscripts to denote elements of the set $\{1,2\}$ and subscripts to index the functions we use in decompositions.) Given a function $a:X\to\F$, we define $a^\alpha$ to be the projection of $a$ to $X^\alpha$: that is, $a^\alpha(x)=a(x)$ if $x\in X^\alpha$ and $a^\alpha(x)=0$ otherwise. We do the same for functions defined on $Y$ and $Z$. Similarly, if $b:Y\times Z\to\F$, then $b^{\beta\gamma}$ is the projection of $b$ to $Y^\beta\times Z^\gamma$, and so on. In particular, if $T:X\times Y\times Z\to\F$ is a tensor, then $T^{\alpha\beta\gamma}$ is the projection of $T$ to $X^\alpha\times Y^\beta\times Z^\gamma$. 

We shall also sometimes use this notation to refer to restrictions rather than projections. For example, if we say that $T=T^{111}\oplus T^{222}$, we mean that $T^{\alpha\be\g}=0$ except if $\a=\be=\g$. In other words, it is sometimes convenient to regard $T^{\a\be\g}$ as defined on $X^\a\times Y^\be\times Z^\g$, and it is sometimes convenient to regard it as defined on all of $X\times Y\times Z$ but supported on $X^\a\times Y^\be\times Z^\g$, and similarly for functions of fewer variables. We hope that no confusion will arise.

\begin{corollary}\label{normalform}
Let $X=X^1\cup X^2$, $Y=Y^1\cup Y^2$ and $Z=Z^1\cup Z^2$ be three finite sets each partitioned into two subsets, and let $T:X\times Y\times Z\to\F$ be a tensor. Suppose that $T$ has a decomposition
\begin{align}\label{rstdecomp}T(x,y,z)=\sum_{i=1}^ra_i(x)b_i(y,z)+\sum_{j=1}^sc_j(y)d_j(x,z)+\sum_{k=1}^te_k(z)f_k(x,y).\end{align}
Then $T$ has such a decomposition with the additional property that for all $\alpha,\beta,\gamma\in\{1,2\}$ the non-zero $a_i^\alpha$ are linearly independent, the non-zero $c_j^\beta$ are linearly independent, and the non-zero $e_k^\gamma$ are linearly independent.
\end{corollary}

\begin{proof}
Applying Lemma \ref{independence} with $V=\F^{X^1}$, $W=\F^{X^2}$, $v_i=a_i^1$, and $w_i=a_i^2$ for each $i$, we obtain a sequence $a_1',\dots,a_r'$ with the same linear span as $a_1,\dots,a_r$ such that the non-zero vectors $(a_i')^1$ are linearly independent and the non-zero vectors $(a_i')^2$ are linearly independent. By Lemma \ref{rewriteterm} we can find functions $b_1',\dots,b_r':Y\times Z\to\F$ such that $\sum_ia_i(x)b_i(y,z)=\sum_ia_i'(x)b_i'(y,z)$ for every $x,y,z$. By symmetry we can rewrite the other two terms in a similar way, and the result is proved.
\end{proof}

We need one further linear algebra lemma.

\begin{lemma}\label{insubspace} Let $U,V,W$ be vector spaces and let $W'$ be a subspace of $W$. Let $u_1,\dots,u_r\in U$ be linearly independent and let $v_1,\dots,v_s\in V$ be linearly independent. Suppose that we have a linear combination $\sum_{i=1}^r\sum_{j=1}^su_i\otimes v_j\otimes w_{ij}$ that belongs to the subspace $U\otimes V\otimes W'$. Then all the vectors $w_{ij}$ belong to the subspace $W'$.
\end{lemma}

\begin{proof}
Suppose not, and let $\phi:W\to\mathbb F$ be a linear functional that vanishes on $W'$ but not on every vector $w_{ij}$. Define $\psi(u\otimes v\otimes w)$ to be $\phi(w)u\otimes v$ and extend this to a linear map $\psi:U\otimes V\otimes W\to U\otimes V$. Then $\psi$ vanishes on $U\otimes V\otimes W'$. However, the image of $\sum_{i=1}^r\sum_{j=1}^su_i\otimes v_j\otimes w_{ij}$ is a non-zero linear combination of the $u_i\otimes v_j$, which are linearly independent, so it is non-zero. This is a contradiction. 
\end{proof}

Now let us adopt our main hypothesis, namely that we have a tensor $T$ as in Corollary \ref{normalform} and that $T=T^{111}\oplus T^{222}$. Suppose also that $T$ has a decomposition as in (\ref{rstdecomp}) above, and that the conclusion of Corollary \ref{normalform} holds for this decomposition. Our hypothesis is equivalent to the statement that $T^{\alpha\beta\gamma}=0$ except if $\alpha=\beta=\gamma$. 

For $\alpha,\beta,\gamma\in\{1,2\}$ let $A^\alpha=\{i:a_i^\alpha\ne 0\}$, let $C^\beta=\{j:c_j^\beta\ne 0\}$, and let $E^\gamma=\{k:e_k^\gamma\ne 0\}$. Then for each $\alpha,\beta,\gamma,x,y,z$, we have that
\[T^{\alpha\beta\gamma}(x,y,z)=\sum_{i\in A^\alpha}a_i^\alpha(x)b_i^{\beta\gamma}(y,z)+\sum_{j\in C^\beta}c_j^\beta(y)d_j^{\alpha\gamma}(x,z)+\sum_{k\in E^\gamma}e_k^\gamma(z)f_k^{\alpha\beta}(x,y).\]

In the next lemma, we shall use bracketed superscripts to denote dependencies and non-bracketed superscripts to denote the parts that a function applies to. So for example, in the statement, the function $p_{ij}^{(\a)\g}$ is defined on $Z^\g$ and depends on $\a$ (because it will be made out of the functions $d_j^{\a\g}$, which are defined on $X^\a\times Z^\g$). 

\begin{lemma}\label{lowrank}
Let $\a,\be,\g$ be not all equal and let $i\in A^\alpha$. Then there exist functions 
$p_{ij}^{(\a)\g}:Z^\gamma\to\F$ and $q_{ik}^{(\a)\be}:Y^\beta\to\F$ such that 
\[b_i^{\beta\gamma}=\sum_{j\in C^\beta}c_j^\beta\otimes p_{ij}^{(\a)\g}+\sum_{k\in E^\g}q_{ik}^{(\a)\be}\otimes e_k^\g,\]
with similar decompositions for $d_j^{\alpha\g}$ and $f_k^{\a\be}$.
\end{lemma}

\begin{proof}
Since the $a_i^\a$ with $i\in A^\a$ are linearly independent, the matrix $(a_i^\alpha(x))$, where $i$ ranges over $A^\alpha$ and $x$ over $X$, has rank $|A^\alpha|$. It follows that we can find for each $i$ a function $h_i^{(\a)}:X\to\F$ such that $\sum_xh_i^{(\a)}(x)a_l^\a(x)=\d_{il}$ for every $l\in A^\a$. Then since $T^{\a\be\g}=0$, we have that
\[0=\sum_xh_i^{(\a)}(x)T^{\a\be\g}(x,y,z)=b_i^{\be\g}(y,z)-\sum_{j\in C^\be}c_j^\be(y)p_{ij}^{(\a)\g}(z)-\sum_{k\in E^\g}q_{ik}^{(\a)\be}(y)e_k^\g(z),\]
where 
\[p_{ij}^{(\a)\g}(z)=-\sum_xh_i^{(\a)}(x)d_j^{\a\g}(x,z)\]
and
\[q_{ik}^{(\a)\be}(y)=-\sum_xh_i^{(\a)}(x)f_k^{\a\be}(x,y).\]
The corresponding results for the functions $d_j^{\a\g}$ and $f_k^{\a\be}$ are proved in the same way.
\end{proof}

Using Lemma \ref{lowrank} we can rewrite the decomposition of $T^{\a\be\g}$ above in the form
\begin{align*}\sum_{i\in A^\a}&\sum_{j\in C^\be}a_i^\a\otimes c_j^\be\otimes p_{ij}^{(\a)\g}+\sum_{i\in A^\a}\sum_{k\in E^\g}a_i^\a\otimes q_{ik}^{(\a)\be}\otimes e_k^{\g}\\
&+\sum_{i\in A^\a}\sum_{j\in C^\be}a_i^\a\otimes c_j^\be\otimes g_{ij}^{(\be)\g}+\sum_{j\in C^\be}\sum_{k\in E^\g}h_{jk}^{\a(\be)}\otimes c_j^\be\otimes e_k^\g\\
&+\sum_{i\in A^\a}\sum_{k\in E^\g}a_i^\a\otimes u_{ik}^{\be(\g)}\otimes e_k^\g+\sum_{j\in C^\be}\sum_{k\in E^\g}v_{jk}^{\a(\g)}\otimes c_j^\be\otimes e_k^\g.\\
\end{align*}
We are still assuming here that $\a,\be$ and $\g$ are not all equal.

Since $T^{\a\be\g}$ is also equal to 0 under this assumption, it follows from Lemma \ref{insubspace} that $p_{ij}^{(\a)\g}+g_{ij}^{(\be)\g}$ is a linear combination of the $e_k^\g$ with $k\in E^\g$, with similar statements for $q_{ik}^{(\a)\be}+u_{ik}^{\be(\g)}$ and for $h_{jk}^{\a(\be)}+v_{jk}^{\a(\g)}$.

We now show that the result is true in the extreme case that $A^1=A^2$, $B^1=B^2$ and $C^1=C^2$. 

\begin{corollary}\label{extremecase}
Suppose that $A^1=A^2=[r]$, $B^1=B^2=[s]$ and $C^1=C^2=[t]$. Then the slice ranks of $T^{111}$ and $T^{222}$ are both at most $\min\{r,s,t\}$.
\end{corollary}

\begin{proof} For this proof, let us adopt the convention that summing over $i$ means summing over $i\in A^\a=A^\be$, and similarly for $j$ and $k$. 

From what we have just proved, with $(\a,\be,\g)=(2,2,1)$, we have for all $i,j$ that $p_{ij}^{(\a=2)1}+g_{ij}^{(\be=2)1}$ is a linear combination of the $e_k^1$, and we have similar conclusions for $q_{ik}^{(\a=2)1}+u_{ik}^{1(\g=2)}$ and $h_{jk}^{1(\be=2)}+v_{jk}^{1(\g=2)}$. Here we are writing $p_{ij}^{(\a=2)1}$ to denote the function $p_{ij}^{(\a)1}$ in the case $\a=1$, and so on. (It would be nice to be able to write the simpler $p_{ij}^{(2)1}$, but then it would not be clear that 2 was the value taken by $\a$.) 

Now recall that for all $\a,\be,\g$, we have that
\[T^{\alpha\beta\gamma}(x,y,z)=\sum_{i}a_i^\alpha(x)b_i^{\beta\gamma}(y,z)+\sum_{j}c_j^\beta(y)d_j^{\alpha\gamma}(x,z)+\sum_{k}e_k^\gamma(z)f_k^{\alpha\beta}(x,y).\]
Substituting the formulae obtained in Lemma \ref{lowrank} for $b_i^{11}$, $d_j^{11}$ and $f_k^{11}$ by taking $(\a,\be,\g)=(2,1,1), (1,2,1)$ and $(1,1,2)$, respectively, we obtain the formula
\begin{align*}T^{111}&=\sum_{i,j}a_i^1\otimes c_j^1\otimes p_{ij}^{(\a=2)1}+\sum_{i,k}a_i^1\otimes q_{ik}^{(\a=2)1}\otimes e_k^1\\ 
&+\sum_{i,j}a_i^1\otimes c_j^1\otimes g_{ij}^{(\be=2)1}+\sum_{j,k}h_{jk}^{1(\be=2)}\otimes c_j^1\otimes e_k^1\\
&+\sum_{i,k}a_i^1\otimes u_{ik}^{1(\g=2)}\otimes e_k^1+\sum_{j,k}v_{jk}^{1(\g=2)}\otimes c_j^1\otimes e_k^1.\\
\end{align*}
The observations in the second paragraph of this proof imply that the right hand side belongs to the linear span of the functions $a_i^1\otimes c_j^1\otimes e_k^1$. From this the result for $T^{111}$ follows. The proof for $T^{222}$ is similar. 
\end{proof}

%\begin{corollary}\label{extremecase}
%Suppose that $A^1=A^2=[r]$, $B^1=B^2=[s]$ and $C^1=C^2=[t]$. Then the slice ranks of $T^{111}$ and $T^{222}$ are both at most $\min\{r,s,t\}$.
%\end{corollary}

%\begin{proof}
%In this case the conclusion of Lemma \ref{lowrank} holds even if $\a=\be=\g$. For instance, if they are all equal to 1, then $i\in A^2$, since $A^1=A^2$, and therefore by Lemma \ref{lowrank} we still have the decomposition
%\[b_i^{11}=\sum_{j\in C^1}c_j^1\otimes p_{ij}^{11}+\sum_{k\in E^1}q_{ik}^{11}\otimes e_k^1.\]
%Therefore the above expression for $T^{\a\be\g}$ is also valid for every $\a,\be,\g$.

%But then Lemma \ref{insubspace} implies that every $p_{ij}^{\be\g}+g_{ij}^{\a\g}$ is in the linear span of the $e_k^\g$, that every $q_{ik}^{\be\g}+u_{ik}^{\be\g}$ is in the linear span of the $c_j^\be$, and that every $h_{jk}^{\a\g}+v_{jk}^{\a\be}$ is in the linear span of the $a_i^\a$. We may therefore write $T^{\a\be\g}$ as a sum of the form
%\[\sum_{i=1}^r\sum_{j=1}^s\sum_{k=1}^t\lambda_{ijk}a_i^\a\otimes c_j^\be\otimes e_k^\g,\]
%which has slice rank at most $\min\{r,s,t\}$. (This is an uninteresting statement when $\a,\be$ and $\g$ are not all equal since then $T^{\a\be\g}=0$, by hypothesis, but it is non-trivial when $\a=\be=\g$.)
%\end{proof}

Since $2\min\{r,s,t\}\leq r+s+t$, we are done in this case.

%\begin{remark} There is a slightly subtle point hidden in the first sentence of the proof of the corollary above that is perhaps worth expanding on in more detail. As stated, Lemma \ref{lowrank} gives us a decomposition
%\[b_i^{11}=\sum_{j\in C^1}c_j^1\otimes p_{ij}^{11}+\sum_{k\in E^1}q_{ik}^{11}\otimes e_k^1\]
%because we can apply it with $\a=2$ and $\be=\g=1$. Examining the proof of Lemma \ref{lowrank} we see that for example $p_{ij}^{11}$ is given by the formula
%\[p_{ij}^{11}(z)=-\sum_xh_i(x)d_j^{21}(x,z),\]
%so there is a ``2-dependence", and there is another one not made explicit in the notation, which is that the functions $h_i$ depend on the functions $a_i^2$. However, all we need from Lemma \ref{lowrank} is the conclusion, which is that there \emph{exist} functions $p_{ij}^{\be\g}$ (and similarly for the other functions) that decompose $b_i^{11}$ as stated. It does not matter where they come from. 

%Similarly, the comment after the lemma, proving for instance that $p_{ij}^{\be\g}+g_{ij}^{\a\g}$ belongs to the linear span of the functions $e_k^\g$, relies only on the hypothesis that the functions $b_i^{\be\g},d_j^{\a\g}$ and $f_k^{\a\be}$ are decomposed in the way stated, and not on any fact about how the decomposition was created. 
%\end{remark}

To do the general case, we reduce to the case covered by Corollary \ref{extremecase} using an inductive argument.

\begin{proof}[Proof of Theorem \ref{main} for 3-tensors]
Suppose now that the hypothesis of Corollary \ref{extremecase} does not hold. Then without loss of generality $a_1^2=0$. Let $P$ be the matrix of a projection to the one-dimensional subspace of $\F^{X}$ generated by $a_1$ such that $P$ vanishes on all functions supported in $X^2$, and let $Q=I-P$. Then 
\[T(x,y,z)=\sum_{x'}P(x,x')T(x',y,z)+\sum_{x'}Q(x,x')T(x',y,z).\]
For every $y,z$, the sum in the first term is a function of $x$, and that function is a multiple of $a_1^1$. Therefore, it can be written in the form $a_1^1(x)b(y,z)$. Also, if $(y,z)\notin Y^1\times Z^1$, then $T(x',y,z)=0$ for every $x'\in X^1$, and therefore the first term vanishes, by the condition that $P$ vanishes on functions supported in $X^2$. It follows that $b$ is supported on $Y^1\times Z^1$.

As for the second term, writing $Qg(x,u_1,\dots,u_m)$ as shorthand for the sum $\sum_{x'}Q(x,x')g(x',u_1,\dots,u_m)$, it is equal to
\[\sum_{i=1}^rQa_i(x)b_i(y,z)+\sum_{j=1}^sc_j(y)Qd_j(x,z)+\sum_{k=1}^te_k(z)Qf_k(x,y).\]
But $Qa_1=0$, so this is a decomposition of $QT$ into $(r-1)+s+t$ pieces. Furthermore, since $PT$ is supported in $X^1\times Y^1\times Z^1$, it follows that $QT$ is also a direct sum. Therefore, by induction on $r+s+t$, $\sigma((QT)^{111})+\sigma((QT)^{222})\leq r-1+s+t$. Since $(PT)^{111}(x,y,z)=a_1^1(x)b(y,z)$ and $(PT)^{222}=0$, it follows that $\sigma(T^{111})+\sigma(T^{222})\leq r+s+t$. 
\end{proof}

\iftrue
\else
Since $2\min\{r,s,t\}\leq r+s+t$, we are done in this case. It is also easy to see that we are done in the other extreme case, where $A^1\cap A^2, B^1\cap B^2$ and $C^1\cap C^2$ are all empty. So now we would like to deal with the intermediate cases. For this we need a further idea and a couple more linear-algebraic lemmas.

The idea is that if we want to prove that $\sigma(T)\geq\sigma(T^{111})+\sigma(T^{222})$, it is sufficient to prove the same statement for a different tensor $U$ provided that $\sigma(U)=\sigma(T)$, $\sigma(U^{111})=\sigma(T^{111})$, and $\sigma(U^{222})=\sigma(T^{222})$. So we are free to apply to $T$ any transformation that preserves these slice ranks. We shall also want to ensure that $U=U^{111}\oplus U^{222}$ in order to apply our results above. This idea is not mathematically essential, but it is very convenient: it allows us to change basis in a certain way that makes the problem easier to think about.

The following lemma resembles Lemma \ref{independence}, both in its statement and its proof.

\begin{lemma}\label{twopartmatrix}
Let $M=\begin{pmatrix}M^1\\ M^2\end{pmatrix}$ be a matrix, thought of as a function of two variables $x\in X$ and $y\in Y$, and let $X^1$ and $X^2$ be the partition of $X$ that corresponds to the partition of the rows of $M$ into the rows of $M^1$ and the rows of $M^2$. Then there are invertible matrices $P,Q$ such that $P$ is of the form $\begin{pmatrix}P^1&0\\ 0&P^2\\ \end{pmatrix}$ with $P^1$ supported in $X^1\times X^1$ and $P^2$ supported in $X^2\times X^2$, and such that $PMQ$ is of the form $\begin{pmatrix}I_{r_1}&0&0\\ 0&0&A\\ 0&I_{r_2}&0\\ 0&0&B\\ \end{pmatrix}$, where the top half of the matrix is supported in $X^1\times Y$, the bottom half is supported in $X^2\times Y$, the non-zero columns of $A$ and the non-zero columns of $B$ are linearly independent, and the set of $y$ for which the $y$th column of $A$ is non-zero is the same as it is for $B$.
\end{lemma}

\begin{proof}
If the columns of $M^1$ and the columns of $M^2$ are linearly independent, then we are done. Otherwise, we proceed by induction on the number of columns. Suppose without loss of generality that the columns of $M^2$ are not linearly independent. Then we can right-multiply by an invertible matrix and obtain a matrix for which $M^2$ has a column of zeros. If the corresponding column of $M^1$ also has zeros (after the transformation), then we can remove that column from the whole matrix and apply induction. 

If the corresponding column of $M^1$ has a non-zero entry, then we can left-multiply by a matrix $\begin{pmatrix}P^1&0\\ 0&I^2\\ \end{pmatrix}$, where $P^1$ is an invertible matrix supported on $X^1\times X^1$ and $I^2$ denotes the identity matrix supported on $X^2\times X^2$, to obtain a matrix such that the given column contains a single entry of 1 and is otherwise 0. (This is because such a matrix can be obtained from $M$ by applying elementary row operations that involve only the rows from $X^1$.) We can then right-multiply by an invertible matrix to ensure that all the other entries in the row that contains the 1 are 0 as well. Applying further row and column operations, we can ensure that the top left-hand entry is 1 and that all other entries in the first row and column of $M$ are 0.  

Now we can remove the first row and column that contain that entry and apply the inductive hypothesis to the matrix that remains, and we obtain the desired conclusion.
\end{proof}

\begin{lemma}\label{invariant}
Let $T:X\times Y\times Z\to\F$ be a 3-tensor as before, with $X$, $Y$ and $Z$ each partitioned into two sets. Let $P:X\times X\to F$ be a matrix of the form $\begin{pmatrix}P^1&0\\ 0&P^2\\ \end{pmatrix}$ with $P^1$ and $P^2$ invertible. Let $U$ be the tensor defined by the formula
\[U(x,y,z)=\sum_{x'}P(x,x')T(x',y,z).\]
Then $U=U^{111}\oplus U^{222}$, $\sigma(U)=\sigma(T)$, $\sigma(U^{111})=\sigma(T^{111})$, and $\sigma(U^{222})=\sigma(T^{222})$.
\end{lemma}

\begin{proof}
Since $P(x,x')$ is zero unless either $x,x'\in X^1$ or $x,x'\in X^2$, it follows from the fact that $T^{\a\be\g}=0$ unless $\a=\be=\g$ that the same is true of $U$. This proves the first assertion.

For the statements about slice rank it suffices to show that $\sigma(U)\leq\sigma(T)$ for any invertible matrix $P$ and any 3-tensor $T$ (that is, not necessarily direct sums), since we can then apply the result to all three tensors, and using the inverse of $P$ we can obtain the reverse inequalities.

But if $T$ has a decomposition
\[T(x,y,z)=\sum_{i=1}^ra_i(x)b_i(y,z)+\sum_{j=1}^sc_j(y)d_j(x,z)+\sum_{k=1}^te_k(z)f_k(x,y),\]
then clearly so does $U$. Indeed, if we set $a_i'(x)=\sum_{x'}P(x,x')a_i(x')$, $d_j'(x,z)=\sum_{x'}P(x,x')d(x',z)$, and $f_k'(x,y)=\sum_{x'}P(x,x')f_k(x',y)$, then
\[U(x,y,z)=\sum_{i=1}^ra_i'(x)b_i(y,z)+\sum_{j=1}^sc_j(y)d_j'(x,z)+\sum_{k=1}^te_k(z)f_k'(x,y),\]
giving us a decomposition of $U$ of the same size. 
\end{proof}

It will also be important to us that the functions $c_j$ and $e_k$ are unchanged when we apply this transformation to $T$.
\fi

\iftrue
\else
The transformations we shall use are natural analogues of elementary row operations applied to matrices. Define a \emph{slice} of a 3-tensor $T$ to be a matrix obtained by holding one of the three variables constant. We shall call it an $x$-slice if $x$ is the variable held constant, and similarly for $y$ and $z$. Define an \emph{elementary $x$-slice operation} to be an operation of one of the following three kinds.
\begin{enumerate}
\item Multiply an $x$-slice by a non-zero scalar.
\item Interchange two $x$-slices.
\item Replace an $x$-slice by the sum of that $x$-slice and a multiple of another $x$-slice.
\end{enumerate}
Of course, we define $y$-slice operations and $z$-slice operations similarly.

\begin{lemma}\label{operations}
Let $T$ be a 3-tensor and let $U$ be obtained from $T$ by means of elementary slice operations. Then $\sigma(U)=\sigma(T)$.
\end{lemma}

\begin{proof}
Since the inverse of a slice operation is a slice operation, it suffices to prove that $\sigma(U)\leq\sigma(T)$. But if $T$ has a decomposition
\[T(x,y,z)=\sum_{i=1}^ra_i(x)b_i(y,z)+\sum_{j=1}^sc_j(y)d_j(x,z)+\sum_{k=1}^te_k(z)f_k(x,y),\]
and $U$ is obtained from $T$ by an elementary slice operation, then we can simply apply that slice operation to each of the rank-1 tensors in the decomposition of $T$ and it is simple to see that we obtain a decomposition of $U$.
\end{proof}

An important remark is that if we apply an $x$-slice operation, this will change the functions $d_j$ and $f_k$ but not the functions $c_j$ and $e_k$.

Now suppose that $T=T^{111}\oplus T^{222}$. Then the diagonal-sum property is preserved by $x$-slice operations provided that if they involve two slices, then those two slices either both come from $X^1$ or both come from $X^2$. 

Suppose once again that $T:X\times Y\times Z\to\F$, that $X, Y$ and $Z$ are partitioned into sets $X^1\cup X^2$, $Y^1\cup Y^2$, and $Z^1\cup Z^2$, respectively, that $T=T^{111}\oplus T^{222}$, and that $T$ has the decomposition
\[T(x,y,z)=\sum_{i=1}^ra_i(x)b_i(y,z)+\sum_{j=1}^sc_j(y)d_j(x,z)+\sum_{k=1}^te_k(z)f_k(x,y).\]
Let $M$ be the $|X|\times r$ matrix whose columns are $a_1,\dots,a_r$: that is, $M_{xi}=(a_i(x))$, for $i=1,2,\dots,r$ and $x\in X$. Let $M^1$ and $M^2$ be the restrictions of $M$ to the sets $x\in X^1$ and $x\in X^2$, respectively, so $M=\begin{pmatrix}M^1\\ M^2\\ \end{pmatrix}$. 
\fi

\iftrue
\else
The effect of an $x$-slice operation on the decomposition of $T$ is to apply the corresponding row-operation to the matrix $A$. For instance, if $x_1\ne x_2$ and we let $U(x_1,y,z)=T(x_1,y,z)+\lambda T(x_2,y,z)$ and $U(x,y,z)=T(x,y,z)$ otherwise, then the effect on each $a_i$ will be to change the value of $a_i(x_1)$ to $a_i(x_1)+\lambda a_i(x_2)$. 

We can write $M=\begin{pmatrix}M^1\\ M^2\\ \end{pmatrix}$ where the rows of $M^1$ correspond to those $x$ that belong to $X^1$ and the rows of $M^2$ correspond to those $x$ that belong to $X^2$. Using elementary row operations that preserve the diagonal-sum property we can put both $M^1$ and $M^2$ into reduced row-echelon form. 

Suppose that we perform the corresponding elementary slice operations on $T$. Then we end up with a tensor $U$ such that $U=U^{111}\oplus U^{222}$, $\sigma(U)=\sigma(T), \sigma(U^{111})=\sigma(T^{111})$, and $\sigma(U^{222})=\sigma(T^{222})$. Moreover, $U$ has a decomposition of the form 
\[\sum_{i=1}^ra_i(x)b_i(y,z)+\sum_{j=1}^sc_j(y)d_j(x,z)+\sum_{k=1}^te_k(z)f_k(x,y)\]
such that the two parts of the matrix $M$ defined above are in reduced row-echelon form.

Each column of $M$ contains zero, one or two leading entries. (This is because $M$ itself is not in reduced row-echelon form, but rather consists of two submatrices each in reduced row-echelon form.) For each column that contains exactly one leading entry, we can perform elementary column operations to ensure that the rest of the row that contains that entry is zero. Since elementary column operations do not affect the linear span of the columns, they do not affect the linear span of the $a_i$, so by Lemma \ref{rewriteterm} we can rewrite the first term of the decomposition so as to ensure that every row that contains a leading entry that is the only leading entry in its column consists only of zeros, apart from that leading entry.

After reordering the rows and columns if necessary, we can say that $M$ has the following structure:
\[\begin{pmatrix}I_{r_1}&0&0&0\\ 0&0&I_{r_3}&A\\ 0&0&0&0\\ 0&I_{r_2}&0&0\\0&0&I_{r_3}&B\\ 0&0&0&0\\ \end{pmatrix}\]
After further column operations, we can reduce $A$ to 0. 

We now continue along the lines of the proof of Lemma \ref{independence}.

Let $P_1\subset X^1$ be the set of $x$ corresponding to the rows of the matrix $I_{r_1}$, and note that for each $x\in P_1$ there is exactly one $i$ such that $a_i$ is 1 at $x$ and 0 everywhere else. Let $H_1$ be the set of these $i$. Then the tensor $\sum_{i\in H_1}a_i(x)b_i(y,z)$ is supported on $P_1\times Y\times Z$. Similarly, let $P_2\subset X^2$ be the set of $x$ corresponding to the rows of the matrix $I_{r_2}$ and define $H_2$ to be the set of corresponding $i$. Then $\sum_{i\in H_2}a_i(x)b_i(y,z)$ is supported on $P_2\times Y\times Z$. Finally, let $P_3=X\setminus(P_1\cup P_2)$, let $H_3=[r]\setminus(H_1\cup H_2)$, and note that for each $i\in H_3$, $a_i$ is supported in $P_3$. Therefore, we can write
\[\sum_{i=1}^ra_i(x)b_i(y,z)=S_1+S_2+S_3\]
with $S_j(x,y,z)=\sum_{i\in H_j}a_i(x)b_i(y,z)$ supported on $P_j\times Y\times Z$ for $j=1,2,3$.
\fi

\iftrue
\else
Let $P$ and $Q$ be invertible matrices given by Lemma \ref{twopartmatrix} and let $U=PT$ (this being shorthand for the equation $U(x,y,z)=\sum_{x'}P(x,x')T(x',y,z)$). Then we have the conclusion of Lemma \ref{invariant}. Let us rename the functions $a_i, d_j$ and $f_j$ so that now they name the functions that are used to decompose $U$, and let us define $M$ as above, but with the new functions $a_i$. Then $MQ$ satisfies the conclusion of Lemma \ref{twopartmatrix}, and since $Q$ is invertible the columns of $MQ$ generate the same subspace as those of $Q$. Therefore, letting $a_1',\dots,a_r'$ be these columns, Lemma \ref{rewriteterm} allows us to rewrite the term $\sum_{i=1}^ra_i(x)b_i(y,z)$ as a term $\sum_{i=1}^ra_i'(x)b_i'(y,z)$ for some functions $b_1',\dots,b_r'$. Again let us rename and drop the dashes, and let us write $M$ instead of $MQ$. 

Lemma \ref{twopartmatrix} tells us that $M$ has the form $\begin{pmatrix}I_{r_1}&0&0\\ 0&0&A\\ 0&I_{r_2}&0\\ 0&0&B\\ \end{pmatrix}$, where the top half and bottom half are $M^1$ and $M^2$. For $i=1,2$ let $P^i$ be the set of $x$ that correspond to the rows of the identity matrix $I_{r_i}$, and let $P^3=X\setminus(P^1\cup P^2)$. Similarly, for $i=1,2$ let $H^i$ be the set of indices of columns of the identity matrix $I_{r_i}$ and let $H^3=[r]\setminus(H^1\cup H^2)$. Then we can write the decomposition of $U$ as
\begin{align*}U(x,y,z)=\sum_{i\in H^1}a_i(x)&b_i(y,z)+\sum_{i\in H^2}a_i(x)b_i(y,z)+\sum_{i\in H^3}a_i(x)b_i(y,z)\\
&+\sum_{j=1}^sc_j(y)d_j(x,z)+\sum_{k=1}^te_k(z)f_k(x,y).\\
\end{align*}
The first two terms are supported in $P^1\times Y\times Z$ and $P^2\times Y\times Z$, respectively. By subtracting off suitable parts from the last two terms and adding them to the first two terms, we can obtain a different decomposition of $U$ such that the first two terms are still supported in $P^1\times Y\times Z$ and $P^2\times Y\times Z$ and the three remaining terms are all supported in $P^3\times Y\times Z$. 

Let us assume that this has been done and that the functions have been renamed. Let $V$ be the sum of the last three terms. Then $V$ is supported in $P^3\times Y\times Z$ and $V=V^{111}\oplus V^{222}$. To see the truth of this last claim, suppose that $(x,y,z)\in (P_3\cap X^\alpha)\times Y^\beta\times Z^\gamma$ and $\alpha,\beta$ and $\gamma$ are not all equal. Then $\sum_{i\in H_\theta}a_i(x)b_i(y,z)=0$ if $\theta=1$ or $\theta=2$, since then $i\in H_1\cup H_2$, which implies that $a_i(x)=0$. It follows that $V(x,y,z)=U(x,y,z)=0$.

By induction on $r+s+t$, as long as $H^1$ and $H^2$ are not both empty,
\[\sigma(V^{111})+\sigma(V^{222})\leq\sigma(V)\leq|H^3|+s+t,\]
which implies that 
\[\sigma(U^{111})+\sigma(U^{222})\leq|H^1|+|H^2|+|H^3|+s+t=r+s+t,\]
and we are done. 

If $H^1$ and $H^2$ are both empty, and the same is true for the matrices that correspond to the other two terms, then the hypotheses of Corollary \ref{extremecase} hold and again we are done. This completes the proof of Theorem \ref{main}.
\fi

\iftrue
\else

Now let us return to the decomposition of $U$, which (after suitable renaming) will be of the form
\[\sum_{i=1}^ra_i(x)b_i(y,z)+\sum_{j=1}^sc_j(y)d_j(x,z)+\sum_{k=1}^te_k(z)f_k(x,y).\]
We can modify this decomposition so as to ensure that the second and third terms are supported on $P_3\times Y\times Z$. Indeed, consider the restriction of the term $\sum_{j=1}^sc_j(y)d_j(x,z)$ to $P_1\times Y\times Z$. If we subtract it from the sum, we obtain a new sum of the same kind, where for each matrix $d_j$ we have set the rows with $x\in P_1$ to zero to obtain a matrix $d_j'$ supported on $(X\setminus P_1)\times Z$. If we now compensate for this by adding the restriction to the tensor $S_1$, we obtain a tensor $S_1'$ that is still supported on $P_1\times Y\times Z$ and therefore for trivial reasons still has a decomposition of the form $S_1(x,y,z)=\sum_{i\in H_1}a_i(x)b_i'(y,z)$. Doing this for both terms and for $P_1$ and $P_2$ we arrive, as claimed at a decomposition with the property claimed.

Having made this adjustment, we claim that the tensor $V$ given by
\[V(x,y,z)=\sum_{i\in H_3}a_i(x)b_i(y,z)+\sum_{j=1}^sc_j(y)d_j(x,z)+\sum_{k=1}^te_k(z)f_k(x,y)\]
is a direct sum supported on $P_3\times Y\times Z$. Indeed, if $(x,y,z)\in (P_3\cap X^\alpha)\times Y^\beta\times Z^\gamma$ and $\alpha,\beta$ and $\gamma$ are not all equal, then $\sum_{i\in H_\theta}a_i(x)b_i(y,z)=0$ if $\theta=1$ or $\theta=2$, since then $i\in H_1\cup H_2$, which implies that $a_i(x)=0$. It follows that $V(x,y,z)=U(x,y,z)=0$.

We now repeat the argument for $V$, this time using $y$-slice operations and aiming to put the second term into a suitable ``standard form". Note that none of these operations will affect the conclusions we have reached for the first term.

\end{document}

If $\a=\be=\g$, then the situation is slightly more complicated. Let us consider the case where they are all equal to 1. Then Lemma \ref{lowrank} still gives us a decomposition if $i\in A^2$, with similar statements for the $d_j^{11}$ and $f_k^{11}$, so we have that
\begin{align*}T^{111}(x,y,z)&=\sum_{i\notin A^2}a_i^1(x)b_i^{11}(y,z)+\sum_{j\notin C^2}c_j^1(y)d_j^{11}(x,z)+\sum_{k\notin E^2}e_k^{11}(z)f_k^1(x,y)\\ 
&+\sum_{i\in A^1\cap A^2}\sum_{j\in C^1}a_i^1(x) c_j^1(y) p_{ij}^{11}(z)+\sum_{i\in A^1\cap A^2}\sum_{k\in E^1}a_i^1(x)q_{ik}^{11}(y)e_k^{1}(z)\\
&+\sum_{i\in A^1}\sum_{j\in C^1\cap C^2}a_i^1(x)c_j^1(y)g_{ij}^{11}(z)+\sum_{j\in C^1\cap C^2}\sum_{k\in E^1}h_{jk}^{11}(x)c_j^1(y)e_k^1(z)\\
&+\sum_{i\in A^1}\sum_{k\in E^1\cap E^2}a_i^\a(x)u_{ik}^{11}(y)e_k^1(z)+\sum_{j\in C^1}\sum_{k\in E^1\cap E^2}v_{jk}^{11}(x)c_j^1(y)e_k^1(z).\\
\end{align*}

%&+\sum_{i\in A^1\cap A^2}a_i^1(x)b_i^{11}(y,z)+\sum_{j\in C^1\cap C^2}c_j^1(y)d_j^{11}(x,z)+\sum_{k\in E^1\cap E^2}e_k^1(z)f_k^{11}(x,y).\\
%\end{align*} 
So far, this shows that $T^{111}$ has slice rank at most $|A^1|+|C^1|+|E^1|$, which is not strong enough, since $|A^1|+|A^2|$ may be bigger than $r$, and similarly for the $C^\be$ and $E^\g$. However, we can use the results we have proved to rewrite the second half of the right-hand side in a more efficient way.

To see this, note first that if we set $\a=2,\be=\g=1$ in Lemma \ref{lowrank}, then for each $i\in A^1\cap A^2$ we obtain an expression for $b_i^{11}$ of the form
\[b_i^{11}=\sum_{j\in C^1}c_j^1\otimes p_{ij}^{11}+\sum_{k\in E^1}q_{ik}^{11}\otimes e_k^1,\]
with similar decompositions for $d_j^{11}$ and $f_k^{11}$ when $j\in C^1\cap C^2$ and $k\in E^1\cap E^2$.

This implies that the second half of the decomposition of $T^{111}$ above is the value at $(x,y,z)$ of a tensor of the form
\begin{align*}
&\sum_{i\in A^1\cap A^2}\sum_{j\in C^1}a_i^1\otimes c_j^1\otimes p_{ij}^{11}+\sum_{i\in A^1\cap A^2}\sum_{k\in E^1}a_i^1\otimes q_{ik}^{11}\otimes e_k^1\\
&+\sum_{i\in A^1}\sum_{j\in C^1\cap C^2}a_i^1\otimes c_j^1\otimes g_{ij}^{11}+\sum_{j\in C^1\cap C^2}\sum_{k\in E^1}h_{jk}^{11}\otimes c_j^1\otimes e_k^1\\
&+\sum_{i\in A^1}\sum_{k\in C^1\cap C^2}a_i^1\otimes u_{ik}^{11}\otimes e_k^1+\sum_{j\in C^1}\sum_{k\in E^1\cap E^2}v_{jk}^{11}\otimes c_j^1\otimes e_k^1 .\\ 
\end{align*}
\fi

\iftrue
\else
\section{The proof for $d$-tensors}

We begin with the notation. Let $X_1,\dots,X_d$ be finite sets, with each $X_i$ being partitioned into two sets $X_i^1$ and $X_i^2$. Let $T:X_1\times\dots\times X_d\to\F$ be a $d$-tensor, and write $T^{\a_1\dots\a_d}$ either for the projection or for the restriction to $X_1^{\a_1}\times\dots\times X_d^{\a_d}$ -- which it means will be clear from the context. We assume that $T$ is the direct sum $T^{11\dots1}\oplus T^{22\dots2}$ -- that is, that all other $T^{\a_1\dots\a_d}$ are zero. 

Given $(x_1,\dots,x_d)\in X_1\times\dots\times X_d$, let us write $\overline{x_i}$ as shorthand for the sequence $(x_1,\dots,x_{i-1},x_{i+1},\dots,x_d)$. Then our other assumption about $T$ will be that it has a decomposition of the form
\[T(x_1,\dots,x_d)=\sum_{j_1=1}^{r_1}a_{1j_1}(x_1)b_{1j_1}(\overline{x_1})+\dots+\sum_{j_d=1}^{r_d}a_{dj_d}(x_d)b_{dj_d}(\overline{x_d}),\]
and our aim will be to show that $\sigma(T^{11\dots 1})+\sigma(T^{22\dots 2})\leq r_1+\dots+r_d$ by finding suitable decompositions.

A simple application of Lemma \ref{rewriteterm} allows us to assume that for each $i$ the functions $a_{i1},\dots,a_{ir_i}$ are linearly independent. As in the proof for 3-tensors, we shall begin by proving the result under the additional assumption that the functions $a_{i1}^1,\dots,a_{ir_i}^1$ are linearly independent and the functions $a_{i1}^2,\dots,a_{ir_1}^2$ are linearly independent. While we are making this assumption, we shall adopt as a convention that any sum over $j_i$ is from 1 to $r_i$.

We shall also write $\overline{\a_i}$ as shorthand for $\a_1,\dots,\a_{i-1},\a_{i+1},\dots,\a_d$. We then have that
\[T^{\a_1\dots\a_d}(x_1,\dots,x_d)=\sum_{j_1}a_{1j_1}(x_1)^{\a_1}b_{1j_1}^{\overline{\a_1}}(\overline{x_1})+\dots+\sum_{j_d}a_{dj_d}^{\a_d}(x_d)b_{dj_d}^{\overline{\a_d}}(\overline{x_d}),\]
where as in the $d=3$ case the superscripts denote restrictions to suitable products of the $X_i^{\a_i}$ (so for example each function $b_{ij_i}^{\overline{\a_i}}$ is defined on $X_1^{\a_1}\times\dots\times X_{i-1}^{\a_{i-1}}\times X_{i+1}^{\a_{i+1}}\times\dots\times X_d^{\a_d}$). 

For the proof, it will be useful to condense the notation further. If $A=\{i_1,\dots,i_k\}\subset[d]$ with $i_1<\dots<i_k$, we shall write $x_A$ for $(x_{i_1},\dots,x_{i_k})$, $\a_A$ for $\a_{i_1},\dots,\a_{i_k}$, $X_A^{\a_A}$ for $X_{i_1}^{\a_{i_1}}\times\dots\times X_{i_k}^{\a_{i_k}}$, $j_A$ for $j_{i_1},\dots,j_{i_k}$, and $[r_A]$ for $[r_{i_1}]\times\dots\times[r_{i_k}]$.

As in the proof for $d=3$, the functions $b_{Aj_A}^{\a_{A^c}}$ in the next two results will depend on $\a_i$ for which $i\in A$. For the time being we shall not pay attention to these dependencies, but then we shall come back and see what the proofs tell us about them.

\begin{lemma}\label{step1}
Suppose that
\[0=\sum_{|A|=k}\sum_{j_A}a_{Aj_A}^{\a_A}(x_A)b_{Aj_A}^{\a_{A^c}}(x_{A^c}),\]
with $a_{Aj_A}^{\a_A}=\mathop{\bigotimes}_{i\in A}a_{ij_i}^{\a_i}$ for each $A$ (but with no such restriction on the $b_{Aj_A}^{\a_{A^c}}$). Now fix some $A$ with $|A|=k$. Then there exist functions $g_{Cj_C}^{\a_C}:X^C\to\F$ such that
\[b_{Aj_A}^{\a_{A^c}}(x_{A^c})=\sum_{|B|=k,B\ne A}\sum_{j_{B\setminus A}}a_{(B\setminus A)j_{B\setminus A}}^{\a_{B\setminus A}}(x_{B\setminus A})g_{(A\cup B)^cj_{(A\cup B)^c}}^{\a_{(A\cup B)^c}}(x_{(A\cup B)^c}).\]
\end{lemma}

\begin{proof} 
By assumption the functions $a_{ij_i}^{\a_i}$ are linearly independent for each $i$, from which it follows that the functions $a_{Aj_A}^{\a_A}$ are linearly independent. From that it follows that we can choose for each $j_A\in[r_A]$ a function $f_{j_A}^{\a_A}$ such that 
\[\sum_{x_A\in X_A^{\a_A}}f_{j_A}^{\a_A}(x_A)a_{Aj_A'}^{\a_A}(x_A)=\delta_{j_Aj_A'}\]
for every $j_A'\in[r_A]$. From that we obtain the identity
\[0=b_{Aj_A}^{\a_{A^c}}(x_{A^c})+\sum_{|B|=k,B\ne A}\sum_{j_B}\sum_{x_A\in X_A^{\a_A}}f_{j_A}^{\a_A}(x_A)a_{Bj_B}^{\a_B}(x_B)b_{Bj_B}^{\a_{B^c}}(x_{B^c}).\]
That is, if we fix $x_{A^c}$ and sum over all $x_A$, we obtain the equality.

We can rewrite this as
\[0=b_{Aj_A}^{\a_{A^c}}(x_{A^c})-\sum_{|B|=k,B\ne A}\sum_{j_{B\setminus A}}a_{(B\setminus A)j_{B\setminus A}}^{\a_{B\setminus A}}(x_{B\setminus A})g_{(A\cup B)^cj_{(A\cup B)^c}}^{\a_{(A\cup B)^c}}(x_{(A\cup B)^c}),\]
where 
\[g_{(A\cup B)^cj_{(A\cup B)^c}}^{\a_{(A\cup B)^c}}(x_{(A\cup B)^c})=-\sum_{j_{A\cap B}}\sum_{x_A\in X_A^{\a_A}}f_{j_A}^{\a_A}(x_A)a_{(A\cap B)j_{A\cap B}}^{\a_{A\cap B}}(x_{A\cap B})b_{Bj_B}^{\a_{B^c}}(x_{B^c}).\]
\end{proof}

\begin{corollary}\label{step2}
Given the assumptions of Lemma \ref{step1}, we have also that
\[0=\sum_{|A|=|B|=k, A\ne B}\sum_{j_{A\cup B}}a_{(A\cup B)j_{A\cup B}}^{\a_{A\cup B}}(x_{A\cup B})g_{(A\cup B)^cj_{(A\cup B)^c}}^{\a_{(A\cup B)^c}}(x_{(A\cup B)^c}).\]
\end{corollary}

\begin{proof}
This follows from substituting the expression we obtained for each $b_{Aj_A}^{\a_{A^c}}$ back into the original formula and noting that $A\cup B$ is the disjoint union of $A$ and $B\setminus A$.
\end{proof}

The next corollary gives us our main inductive step.

\begin{corollary}\label{inductivestep}
Again given the assumptions of Lemma \ref{step1}, there exist functions $b_{Dj_D}^{\a_{D^c}}$ for each $|D|=k+1$ such that
\[0=\sum_{|D|=k+1}\sum_{j_D}a_{Dj_D}^{\a_D}(x_D)b_{Dj_D}^{\a_{D^c}}(x_{D^c}).\]
\end{corollary}

\begin{proof}
Each pair $(A,B)$ of distinct sets of size $k$ has a union of size at least $k+1$, which therefore contains some set $D$ of size $k+1$. Let us therefore partition these pairs into classes $\mathcal A_D$ in such a way that if $(A,B)\in\mathcal A_D$, then $D\subset A\cup B$. Then Corollary \ref{step2} implies that
\[0=\sum_{|D|=k+1}\sum_{(A,B)\in\mathcal A_D}\sum_{j_{A\cup B}}a_{(A\cup B)j_{A\cup B}}^{\a_{A\cup B}}(x_{A\cup B})g_{(A\cup B)^cj_{(A\cup B)^c}}^{\a_{(A\cup B)^c}}(x_{(A\cup B)^c}).\]
Now for each $D$ of size $k+1$, let us set
\[b_{Dj_D}^{\a_{D^c}}(x_{D^c})=\sum_{(A,B)\in\mathcal A_D}a_{((A\cup B)\setminus D)j_{(A\cup B)\setminus D}}^{\a_{(A\cup B)\setminus D}}(x_{(A\cup B)\setminus D})g_{(A\cup B)^cj_{(A\cup B)^c}}^{\a_{(A\cup B)^c}}(x_{(A\cup B)^c}).\]
The result then follows.
\end{proof}

Now let us see what these results and their proofs tell us about the decomposition of $T^{11\dots 1}$.

Note first that if we set $k=1$ and $A=\{i\}$ in Lemma \ref{step1} and if we let $\a_j=1$ if $j\ne i$ and $\a_i=2$, then we obtain a decomposition 
\[b_{ij_i}^{\overline{\a_i}=1}(\overline{x_i})=\sum_{i'\ne i}\sum_{j_{i'}}a_{i'j_{i'}}^{\a_{i'}}(x_{i'})g_{\overline{ii'}j_{\overline{ii'}}}^{\a_{\overline{ii'}}}(x_{\overline{ii'}})\]

It is now time to think about dependencies. Note first from the formula at the end of the proof of Lemma \ref{step1} that $g_{(A\cup B)^cj_{(A\cup B)^c}}^{\a_{(A\cup B)^c}}$ depends on $\a_{A\cap B}$, in the sense that it is partly made out of the functions $a_{ij_i}^{\a_i}$ with $i\in A\cap B$. It is also made out of the functions $b_{Bj_B}^{\a_{B^c}}$, and therefore inherits any dependencies that these might have.

Turning to the formula at the end of the proof of Corollary \ref{inductivestep}, we therefore find that $b_{Dj_D}^{\a_{Dj_D}}$ depends on $\a_E$, where $E=\bigcup_{(A,B)\in\mathcal A_D}(A\cap B)$.

\section{The proof for $d$-tensors}

The generalization to $d$-tensors of the argument just given is fairly straightforward. The two main difficulties are that we need to use a more efficient notation and that the generalizations of Lemma \ref{lowrank} and Corollary \ref{extremecase} are slightly more involved. But neither difficulty is a substantial one, and we will be able to import many of the lemmas from the previous section unchanged into this one.

We begin with the notation. Let $X_1,\dots,X_d$ be finite sets, with each $X_i$ being partitioned into two sets $X_i^1$ and $X_i^2$. Let $T:X_1\times\dots\times X_d\to\F$ be a $d$-tensor, and write $T^{\a_1\dots\a_d}$ either for the projection or for the restriction to $X_1^{\a_1}\times\dots\times X_d^{\a_d}$ -- which it means will be clear from the context. We assume that $T$ is the direct sum $T^{11\dots1}\oplus T^{22\dots2}$ -- that is, that all other $T^{\a_1\dots\a_d}$ are zero. 

Given $(x_1,\dots,x_d)\in X_1\times\dots\times X_d$, let us write $\overline{x_i}$ as shorthand for the sequence $(x_1,\dots,x_{i-1},x_{i+1},\dots,x_d)$. Then our other assumption about $T$ will be that it has a decomposition of the form
\[T(x_1,\dots,x_d)=\sum_{j=1}^{r_1}a_{1j}(x_1)b_{1j}(\overline{x_1})+\dots+\sum_{j=1}^{r_d}a_{dj}(x_d)b_{dj}(\overline{x_d}),\]
and our aim will be to show that $\sigma(T^{11\dots 1})+\sigma(T^{22\dots 2})\leq r_1+\dots+r_d$ by finding suitable decompositions.

A simple application of Lemma \ref{rewriteterm} allows us to assume that for each $i$ the functions $a_{i1},\dots,a_{ir_i}$ are linearly independent. As in the proof for 3-tensors, we shall begin by proving the result under the additional assumption that the functions $a_{i1}^1,\dots,a_{ir_i}^1$ are linearly independent and the functions $a_{i1}^2,\dots,a_{ir_1}^2$ are linearly independent. 

Before we do that, let us observe that Lemma \ref{insubspace} generalizes very straightforwardly. In fact, the case $d=2$ clearly implies the lemma for general $d$, but since we have proved it for $d=3$, we shall slightly artificially use that instead.

\begin{lemma}\label{insubspace2}
Let $d>3$, let $U_1,\dots,U_d$ be vector spaces and let $V_d$ be a subspace of $U_d$. For each $i<d$ let $u_{i1},\dots,u_{r_i}$ be a linearly independent subset of $U_i$. Suppose that we have a linear combination 
\[\sum_{i_1=1}^{r_1}\dots\sum_{i_{d-1}=1}^{r_{d-1}}u_{i_1}\otimes\dots\otimes u_{i_{d-1}}\otimes v_{i_1\dots i_{d-1}}\]
that belongs to the subspace $U_1\otimes\dots\otimes U_{d-1}\otimes V_d$. Then all the vectors $v_{i_1\dots i_{d-1}}$ belong to $V_d$.
\end{lemma}

\begin{proof}
Apply Lemma \ref{insubspace} with $U=U_1\otimes\dots\otimes U_{d-2}$, $V=U_{d-1}$, $W=U_d$, and $W'=V_d$. The hypotheses are easily seen to hold, and therefore so does the conclusion.
\end{proof}

We now prove a result that generalizes Corollary \ref{extremecase}. Its proof will use a generalization of Corollary \ref{lowrank}, but it is less convenient to separate the result into two parts than it was in the previous section.

\begin{lemma} \label{extremecaseford}
With all the assumptions just made, $\sigma(T^{11\dots 1})$ and $\sigma(T^{22\dots 2})$ are both at most $\min_i r_i$.
\end{lemma}

\begin{proof}
Suppose that $\a_1,\dots,\a_d$ are not all the same. Then
\[0=\sum_{j=1}^{r_1}a_{1j}^{\a_1}(x_1)b_{1j}^{\overline{\a_1}}(\overline{x_1})+\dots+\sum_{j=1}^{r_d}a_{dj}^{\a_d}(x_d)b_{dj}^{\overline{\a_d}}(\overline{x_d}).\]
Since the functions $a_{11}^{\a_1},\dots,a_{1r_1}^{\a_1}$ are linearly independent, we can find for each $k\in\{1,2,\dots,r_1\}$ a function $h_k^{(\a_1)}:X_1^{\a_1}\to\F$ such that for every $j\in\{1,2,\dots,r_1\}$ we have that $\sum_{x\in X_1^{\a_1}}h_k^{(\a_1)}(x)a_{1j}^{\a_1}(x)=\d_{kj}$. It follows that 
\[b_{1k}^{\overline{\a_1}}(\overline{x_1})=-\sum_{i=2}^d\sum_{j=1}^{r_i}a_{ij}^{\a_i}(x_i)\sum_{x_1\in X_1^{\a_1}}h_k^{(\a_1)}(x_1)b_{ij}^{\overline{\a_i}}(\overline{x_i}).\]
Note that $\sum_{x_1\in X_1^{\a_1}}h_k^{(\a_1)}(x_1)b_{ij}^{\overline{\a_i}}(\overline{x_i})$ is a function of $x_2,\dots,x_{i-1},x_{i+1},\dots,x_d$ that depends on $\a_1$. 

It follows that if $\a_1,\dots,\a_d$ are not all the same, then for every $k\leq r_1$ there are functions $p_{hjk}^{(\a_1)\overline{\a_1},\overline{\a_h}}$, $h=2,\dots,k$, $j=1,2,\dots,r_h$, such that
\[b_{1k}^{\overline{\a_1}}(\overline{x_1})=\sum_{h=2}^d\sum_{j=1}^{r_h}a_{hj}^{\a_h}(x_h)p_{hjk}^{(\a_1)\overline{\a_1}\overline{\a_h}}(\overline{x_1},\overline{x_h}),\]
where we have extended our notation so that $\overline{\a_1\a_h}$ is shorthand for 
\[\a_2\a_3\dots\a_{h-1}\a_{h+1}\dots\a_d\] 
and $(\overline{x_1},\overline{x_h})$ is shorthand for 
\[(x_2,x_3,\dots,x_{h-1},x_{h+1},\dots,x_d).\] 
More generally, we have for each $b_{ik}^{\overline{\a_i}}$ a decomposition
\[b_{ik}^{\overline{\a_i}}(\overline{x_i})=\sum_{h\ne i}\sum_{j=1}^{r_h}a_{hj}^{\a_h}(x_h)p_{hjk}^{(\a_i)\overline{\a_i}\overline{\a_h}}(\overline{x_i},\overline{x_h}).\]

It follows that for every $\a_1,\dots,\a_d$ that are not all equal, we have the formula
\[T^{\a_1,\dots,\a_d}(x_1,\dots,x_d)=\sum_{i=1}^d\sum_{j=1}^{r_i}\sum_{h\ne i}\sum_{l=1}^{r_h}a_{ij}^{\a_i}(x_i)a_{hl}^{\a_h}(x_h)p_{hjl}^{(\a_i)\overline{\a_i\a_h}}(\overline{x_i},\overline{x_h})\]

The main thing to take away from this formula is that it expresses $T^{\a_1,\dots,\a_d}$ as a sum of products of functions, and each product takes two of the one-variable functions $a_{ij}$ from different classes (that is, it takes $a_{ij}$ and $a_{hl}$ with $i\ne h$) and a further function of all the remaining variables. This is a direct generalization of the statement that immediately follows Lemma \ref{lowrank}, where the further function was also of one variable. Here things are more complicated because the further function is of more than one variable, which means that the argument is not yet finished. 

Note that if $(\a_1,\dots,\a_d)=(1,1,\dots,1)$, then we still have the decomposition
\[b_{ik}^{\overline{\a_i}}(\overline{x_i})=\sum_{h\ne i}\sum_{j=1}^{r_h}a_{hj}^{1}(x_h)p_{hjk}^{(\a_i=2)\overline{\a_i}\overline{\a_h}}(\overline{x_i},\overline{x_h})\]
that comes from considering the sequence $(\be_1,\dots,\be_d)$ where $\be_j=1$ if $j\ne i$ and $\be_i=2$. 

What we shall do now is collect together terms of the same type (that is, ones where the two one-variable functions are the same), so that we have a formula of the kind
\[T^{\a_1,\dots,\a_d}(x_1,\dots,x_d)=\sum_{1\leq i<h\leq d}\sum_{j=1}^{r_i}\sum_{l=1}^{r_h}a_{ij}^{\a_i}(x_i)a_{hl}^{\a_h}(x_h)q_{ihjl}^{(\a_i\a_h)\overline{\a_i\a_h}}(\overline{x_i},\overline{x_h}),\]
which is valid for every choice of $\a_1,\dots,\a_d\in\{1,2\}$ that are not all equal. Here,
\[q_{ihjl}^{(\a_i\a_h)\overline{\a_i\a_h}}(\overline{x_i},\overline{x_h})=p_{hjl}^{(\a_i)\overline{\a_i\a_h}}(\overline{x_i},\overline{x_h})+p_{ijl}^{(\a_h)\overline{\a_i\a_h}}(\overline{x_i},\overline{x_h}).\]
If $\a_1=\dots=\a_d=1$, then we have instead the formula
\[T^{11\dots 1}(x_1,\dots,x_d)=\sum_{1\leq i<h\leq d}\sum_{j=1}^{r_i}\sum_{l=1}^{r_h}a_{ij}^{1}(x_i)a_{hl}^{1}(x_h)q_{ihjl}^{(\a_i=\a_h=2)\overline{\a_i\a_h}}(\overline{x_i},\overline{x_h}),\]
where
\[q_{ihjl}^{(\a_i=\a_h=2)\overline{\a_i\a_h}}(\overline{x_i},\overline{x_h})=p_{hjl}^{(\a_i=2)\overline{\a_i\a_h}}(\overline{x_i},\overline{x_h})+p_{ijl}^{(\a_h=2)\overline{\a_i\a_h}}(\overline{x_i},\overline{x_h}).\]

We have passed from a decomposition into terms that each involve one one-variable function to a decomposition where each term involves two one-variable functions. The idea now is to use a similar process to keep increasing the number of one-variable functions involved in each term until we end up showing that $T^{\a_1\dots\a_d}$ is a linear combination of terms of the form $a_{1i_1}^{\a_1}\otimes\dots\otimes a_{di_d}^{\a_d}$ -- that is, rank-1 tensors formed out of the one-variable functions coming from the original decomposition of $T$. This will prove that each $T^{\a_1\dots\a_d}$ has slice rank at most $\min\{r_1,\dots,r_d\}$, which, as in the 3D case, proves the result under the additional independence assumption.

Let us see how to obtain an expression for $q_{ihjl}^{(\a_i\a_h)\overline{\a_i\a_l}}$ when $\a_1,\dots,\a_d$ are not all the same. The functions $a_{ij}^{\a_i}\otimes a_{hl}^{\a_h}$ are linearly independent, so we can find for each $k\in\{1,2,\dots,r_i\}$ and $m\in\{1,2,\dots,r_h\}$ a function $f_{km}^{\a_i\a_h}:X_i^{\a_i}\otimes X_h^{\a_h}\to\F$ such that for every $j\in\{1,2,\dots,r_i\}$ and $l\in\{1,2,\dots,r_h\}$ we have that
\[\sum_{x_i\in X_i^{\a_i}}\sum_{x_h\in X_h^{\a_h}}f_{km}^{\a_i\a_h}(x_i,x_h)a_{ij}^{\a_i}(x_i)a_{hl}^{\a_h}(x_h)=\d_{kj}\d_{ml}.\]
It follows that 
\[\sum_{x_i\in X_i^{\a_i}}\sum_{x_h\in X_h^{\a_h}}f_{km}^{\a_i\a_h}(x_i,x_h)T^{\a_1\dots\a_d}(x_1,\dots,x_d)\]
is equal to $q_{ihkm}^{(\a_i\a_h)\overline{\a_i\a_h}}(\overline{x_i},\overline{x_h})$ plus a sum of further terms, each of which is of the form
\[\sum_{x_i\in X_i^{\a_i}}\sum_{x_h\in X_h^{\a_h}}f_{km}^{\a_i\a_h}(x_i,x_h)a_{uj}^{\a_u}(x_u)a_{vl}^{\a_v}(x_v)q_{uvjl}^{(\a_u\a_v)\overline{\a_u\a_v}}(\overline{x_u},\overline{x_v}),\]
where $u\ne v$ and the pair $\{u,v\}$ is not equal to the pair $\{i,h\}$. 

If $\{u,v\}$ and $\{i,h\}$ are disjoint, then the above term is equal to 
\[a_{uj}^{\a_u}(x_u)a_{vl}^{\a_v}(x_v)s_{ihuvkmjl}^{(\a_i\a_h\a_u\a_v)\overline{\a_i\a_h\a_u\a_v}}(\overline{x_i},\overline{x_h},\overline{x_u},\overline{x_v}),\]
where
\[s_{ihuvkmjl}^{(\a_i\a_h\a_u\a_v)\overline{\a_i\a_h\a_u\a_v}}(\overline{x_i},\overline{x_h},\overline{x_u},\overline{x_v})=\sum_{x_i\in X_i^{\a_i}}\sum_{x_h\in X_h^{\a_h}}f_{km}^{\a_i\a_h}(x_i,x_h)q_{uvjl}^{(\a_u\a_v)\overline{\a_u\a_v}}(\overline{x_u},\overline{x_v}).\]
If, say, $u\ne i$ but $v=h$, then the term is equal to
\[a_{uj}^{\a_u}(x_u)r_{ihukmj}^{(\a_i\a_h\a_u)\overline{\a_i\a_h\a_u}}(\overline{x_i},\overline{x_h},\overline{x_u}),\]
where
\[r_{ihukmj}^{(\a_i\a_h\a_u)\overline{\a_i\a_h\a_u}}(\overline{x_i},\overline{x_h},\overline{x_u})=\sum_{x_i\in X_i^{\a_i}}\sum_{x_h\in X_h^{\a_h}}f_{km}^{\a_i\a_h}(x_i,x_h)a_{hl}^{\a_h}(x_h)q_{uhjl}^{(\a_u\a_h)\overline{\a_u\a_h}}(\overline{x_u},\overline{x_h})\]

Such terms are of two types, according to whether $\{u,v\}$ and $\{i,h\}$ intersect or are disjoint. In order to understand them without too many indices, let us consider the two expressions
\[\sum_{x_1\in X_1}\sum_{x_2\in X_2}f(x_1,x_2)a(x_2)a'(x_3)b(x_1,x_4,x_5,\dots,x_d)\]
and
\[\sum_{x_1\in X_1}\sum_{x_2\in X_2}f(x_1,x_2)a'(x_3)a''(x_4)b(x_1,x_2,x_5,x_6,\dots,x_d).\]
The first takes the form 
\[a'(x_3)c(x_4,x_5,\dots,x_d)\]
and the second takes the form
\[a'(x_3)a''(x_4)c(x_5,x_6,\dots,x_d).\]
But actually the second form is a special case of the first, so in both cases we have $a'(x_3)$ multiplied by a function of $(x_4,x_5,\dots,x_d)$. 

It follows that each of the further terms mentioned above is of the form $a_{uj}^{\a_u}(x_u)$ multiplied by a function of $(\overline{x_i},\overline{x_h},\overline{x_u})$ (where this $u$ could be either the $u$ or the $v$ from the expression above and the function may depend on $\a_i$, $\a_h$, and $\a_u$). That is, we end up with an expression of the form
\begin{align*}&T^{\a_1\dots\a_d}(x_1,\dots,x_d)\\
&=\sum_{1<i<h<u\leq d}\sum_{j=1}^{r_i}\sum_{l=1}^{r_h}\sum_{m=1}^{r_u}a_{ij}^{\a_i}(x_i)a_{hl}^{\a_h}(x_h)a_{um}^{\a_m}(x_u)r_{ihujlm}^{(\a_i\a_h\a_u)\overline{\a_i\a_h\a_u}}(\overline{x_i},\overline{x_h},\overline{x_u}).\\
\end{align*}

Repeating this argument we end up with an expression for $T^{\a_1\dots\a_n}$ of the form
\[\sum_{i=1}^d\sum_{j_1=1}^{r_1}\dots\sum_{j_{i-1}=1}^{r_{i-1}}\sum_{j_{i+1}=1}^{r_{i+1}}\dots\sum_{j_d=1}^{r_d}a_{1j_1}^{\a_1}\otimes\dots\otimes a_{i-1,j_{i-1}}^{\a_{i-1}}\otimes p_{{\overline{i}}\,{\overline{j_i}}}^{(\overline{\a_i})\a_i}\otimes a_{i+1,j_{i+1}}^{\a_{i+1}}\otimes\dots\otimes\a_{dj_d}^{\a_d}.\]
Here, $p_{{\overline{i}}\,{\overline{j_i}}}^{(\overline{\a_i})\a_i}$ is shorthand for $p_{12\dots (i-1)(i+1)\dots d\ j_1\dots j_{i-1}j_{i+1}\dots j_d}^{(\a_1\dots\a_{i-1}\a_{i+1}\dots\a_d)\ \a_i}$.

By Lemma \ref{insubspace2} we have that $p_{{\overline{i}}\,{\overline{j_i}}}^{(\overline{\a_i})\a_i}$ belongs to the subspace spanned by the functions $a_{ij_i}^{\a_i}$, still under the assumption that the $\a_i$ are not all the same.

\end{proof}

%\[\sum_{i_1=1}^{r_1}\dots\sum_{i_d=1}^{r_d}\lambda_{i_1\dots i_d}a_{i_1}^{\a_1}\otimes\dots\otimes a_{i_d}^{\a_d}.\]
%This proves that $\sigma(T^{11\dots1})$ and $\sigma(T^{22\dots 2})$ are both at most $\min\{r_1,\dots,r_d\}$. \end{proof}

The deduction of the general case from this special case works in exactly the same way that it worked for 3-tensors. We sketch the argument very briefly.

\begin{proof}[Proof of Theorem \ref{main}]
Suppose without loss of generality that $a_{11}^2=0$. Let $P$ and $Q$ be as in the proof when $d=3$. Then as in that proof, $PT$ is of the form 
\[PT(x_1,\dots,x_d)=a_{11}^1(x_1)b(x_2,\dots,x_d),\]
and is supported in $X_1^1\times\dots\times X_d^1$. Also, 
\[QT(x_1,\dots,x_d)=\sum_{j=1}^{r_1}Qa_{1j}(x_1)b_{1j}(\overline{x_1})+\dots+\sum_{j=1}^{r_d}a_{dj}(x_d)Qb_{dj}(\overline{x_d}),\]
and since $Qa_{11}=0$, it follows that $QT$ is a direct sum of slice rank at most $(r_1-1)+r_2+\dots+r_d$. We are therefore done by induction.
\end{proof}
\fi

\section{Further remarks and questions}

There are other basic statements about matrix rank that do not generalize to slice rank for higher-degree tensors. For instance, it is not true in general that $\sigma(S\otimes T)=\sigma(S)\sigma(T)$. Indeed, if one takes three reasonably generic $n\times n\times n$ slice-rank-1 tensors with slices in different directions -- that is, of the kind $a(x)b(y,z)$, $c(y)d(x,z)$, and $e(z)f(x,y)$ -- then their tensor product will tend to have large slice rank. For instance, if $a, e$ and $f$ are all equal to the standard basis vector $e_1$ and $b, d$ and $f$ are all equal to the identity matrix, then the tensor product of the three tensors is equivalent to the so-called matrix multiplication tensor, which has rank $n^2$ (see \cite[Remark 4.9]{BCCGNSU}). And for an example in the other direction, if $T:\F_3^3\to\F_3$ is the characteristic function of the set $\{(x,y,z)\in\F_3^3:x+y+z=0\}$, then it has slice rank 3. (To see this, observe that if not, then it has a decomposition into two functions of slice rank 1, so without loss of generality there is no function of type $e(z)f(x,y)$ involved in the decomposition. But if we then fix $z$, we obtain a matrix of rank 2, but it is also a permutation matrix so it has rank 3, a contradiction.) However, the $n$th tensor power of $T$ can be thought of as the characteristic function of the set $\{(x,y,z)\in(\F_3^n)^3:x+y+z=0\}$, which, as the polynomial method shows, has slice rank exponentially smaller than $3^n$.

A special case of Theorem \ref{main} is that $\sigma(S\otimes T)=\sigma(S)\sigma(T)$ when $S$ is a diagonal tensor, so we obtain equality for this case, but we know in advance that the argument cannot be simple enough to generalize to all tensor products. 

Another related question is a long-standing conjecture of Strassen that tensor rank was additive for direct sums, which, despite being true in a number of special cases, was eventually \emph{dis}proved by Shitov in 2017, who found a highly non-obvious counterexample \cite{shitov}. 

We conclude with three questions. The first is whether there is a simultaneous generalization of the main theorem of this paper and of the result of Sawin and Tao mentioned earlier. To make this question more precise, suppose that $X_i$ is partitioned into sets $X_{i1},\dots,X_{ir_i}$ for each $i$. Define the \emph{block support} of a tensor $T:X_1\times\dots\times X_d$ to be the set of $(j_1,\dots,j_d)$ such that $T$ restricted to the block $X_{1j_1}\times\dots\times X_{dj_d}$ is not identically zero. Define a \emph{block slice} of $T$ to be the restriction of $T$ to a set of the form
\[X_1\times\dots\times X_{h-1}\times X_{hj}\times X_{h+1}\times\dots\times X_d.\]
Call a block  $X_{1j_1}\times\dots\times X_{dj_d}$ \emph{maximal} if $(j_1,\dots,j_d)$ is a maximal element of the block support.

If the non-zero blocks of $T$ are covered by some set of block slices, it is trivial that the slice rank of $T$ is at most the sum of the slice ranks of those block slices. However, sometimes we can improve on this bound. For instance, suppose that the block support of a 3-tensor $T$ is contained in three planes, and contains the intersection of those three planes. Suppose also that the block corresponding to that intersection has high slice rank $r$, and that if that block is removed, then the three block slices have small slice rank $s$. With a suitable example like this, one can arrange that the sum of the slice ranks of block slices that cover the non-zero blocks is minimized in the obvious way, which gives an upper bound of at least $3r$. But one can obtain a better upper bound of $r+3s$ by first decomposing the block at the intersection and then decomposing the rest of the slices.

With that example in mind, let us define a \emph{partial block slice} to be the restriction of $T$ to a union of blocks that forms a subset of a block slice. 

\begin{question}Let $T$ be a $d$-tensor as above and let $S$ be its block support. Does it follow that the slice rank of $T$ is at least the minimum of the sum of the slice ranks of a set of partial block slices that cover all the maximal blocks of $T$?
\end{question}

A positive answer to that question may be too much to hope for, in which case a much weaker preliminary question one might ask is whether if all non-zero blocks have slice rank at least $r$, and if $m$ block slices are needed to cover the maximal blocks, then the slice rank of $T$ is at least $mr$. 
\smallskip

Another obvious question is the following.
\begin{question} Is partition rank additive for direct sums?
\end{question}
It seems reasonable to guess that the answer is no, since the proof just given for slice rank appears to fail quite badly. But that is a pure guess, and it might not be a simple matter to find a counterexample. Naslund showed that if an appropriate extra step is added to Tao's proof of Lemma \ref{diagonal}, then it can be made to yield the stronger result that the partition rank of a diagonal tensor is also equal to the number of non-zero entries \cite{naslund}, so diagonal tensors do not give counterexamples. 

Finally, we ask a more open-ended question.
\begin{question} Does Theorem \ref{main} have any interesting combinatorial applications?
\end{question}
The answer to this is not obvious, given that up to now combinatorial applications have tended to be of the result for diagonal tensors (that is, of Lemma~\ref{diagonal}). 

We do not have a promising suggestion for how to apply the result, but can at least point out one constraint on what a genuine application would need to look like. Suppose that $T_1,\dots,T_m$ are tensors and that the result of Sawin and Tao can be used to show that $\sigma(T_i)\geq r_i$. It then follows easily that $\sigma(T_1\oplus\dots\oplus T_m)\geq r_1+\dots+r_m$. (We observed this in the introduction in the special case where $T_1=\dots=T_m=\e$.) Therefore, an application of the main result of this paper would have to be to tensors $T_1,\dots,T_m$ to which the approach of Sawin and Tao does not apply, which in practice, given the current state of knowledge, means tensors for which we probably do not know how to calculate their slice rank. 

That refers to applications that use direct sums of \emph{specific} tensors. Another possibility might be an argument in which tensors $T_1,\dots,T_m$ are defined in terms of some unknown objects (such as subsets of a finite group) that satisfy certain hypotheses that are used to derive lower bounds for the slice ranks $\sigma(T_i)$. However, for the result of this paper to be used in an essential way, there would still be constraints on the nature of the derivation.

Just before this result was posted, an interesting preprint appeared by Sauermann, who for the first time proved a combinatorial result using a lower bound for the slice rank of a non-diagonal tensor \cite{sauermann}: to obtain the lower bound she relied on the approach of Sawin and Tao. That at least suggests that there is value in extending the known methods for calculating slice rank.

\end{document}